\documentclass[11pt,a4paper]{article}
\usepackage[T1]{fontenc}
\usepackage[utf8]{inputenc}
\usepackage{authblk}
\usepackage{amssymb, amsmath}
\usepackage{amsthm, amsfonts,mathrsfs}
\usepackage[T1]{fontenc}
\usepackage{times}
\usepackage{graphicx}
\usepackage{subfigure}
\usepackage{cite}
\usepackage{hyperref}
\usepackage{epsfig}
\usepackage{color}
\usepackage[all]{xypic}
\hypersetup{colorlinks=true,citecolor=blue,linkcolor=blue}
\newtheorem{thm}{Theorem}[section]
\newtheorem{coro}[thm]{Corollary}

\newtheorem{definition}[thm]{Definition}
\newtheorem{lemma}[thm]{Lemma}
\newtheorem{prop}[thm]{Proposition}
\newtheorem{remark}[thm]{Remark}

\numberwithin{equation}{section}

\newcommand{\ep}{\epsilon}
\newcommand{\al}{\alpha}

\newcommand{\pa}{\partial}
\newcommand{\ue}{u^\epsilon}
\newcommand{\uek}{u^\epsilon_k}
\newcommand{\vek}{v^\epsilon_k}
\newcommand{\fek}{f^\epsilon_k}
\newcommand{\gek}{\gamma^\ep_k}
\newcommand{\wek}{w^{\epsilon,k}}

\newcommand{\bube}{\bar{U}^\epsilon}
\newcommand{\bvbe}{\bar{V}^\epsilon}

\newcommand{\xe}{\xi^\epsilon}

\newcommand{\intot}{\int_{0}^{t}}
\newcommand{\intobtxl}{\int_{0}^{T}\int_{|x|<1}}
\newcommand{\intobtxg}{\int_{0}^{T}\int_{|x|\geq 1}}
\newcommand{\intos}{\int_{0}^{s}}
\newcommand{\intotxl}{\int_{0}^{t}\int_{|x|<1}}
\newcommand{\intotxg}{\int_{0}^{t}\int_{|x|\geq 1}}
\newcommand{\supt}{\sup\limits_{0\leq t\leq T}}

\newcommand{\supe}{\sup\limits_{0<\epsilon\leq 1}}

\newcommand{\ube}{U^\epsilon}
\newcommand{\vbe}{V^\epsilon}
\newcommand{\re}{\rho^\epsilon}
\title{\textbf{On the small mass limit of stochastic wave equation driven by cylindrical stable process }\thanks{This work is supported by a NSFC Grant No.  12371243.}}

\author{Qingming Zhao\thanks{qingming.zhao@smail.nju.edu.cn}}
\author{Xueru Liu \thanks{dg21210009@smail.nju.edu.cn}}
\author{Wei Wang \thanks{Corresponding author: wangweinju@nju.edu.cn}}
\affil{School of Mathematics, Nanjing University, Nanjing 210093, P. R. China}

\date{} 

\begin{document}
	\maketitle

	
		\noindent{\small{\hspace{1.1cm} }}\\
		\\
		\noindent \textbf{Abstract~~~}   We explore the small mass limit of a stochastic wave equation (SWE) driven by cylindrical $\al$-stable noise, where $\alpha\in(1,2)$, and prove that it converges to a stochastic heat equation. We establish its well-posedness, and in particular, the c\`adl\`ag property, which is not trivial in the infinite dimensional case. Using a splitting technique, we decompose the velocity component into three parts, which gives convenience to the moment estimate. We show the tightness of solution of SWE by verifying the infinite dimensional version of Aldous condition. After these preparation, we pass the limit and derive the approximation equation.
		\\[2mm]
		\textbf{Keywords~~~}{Cylindrical stable process}, {Smoluchowski--Kramers approximation}, {Averaging}, {Singular pertubation}
		\\[2mm]
		\\
		\textbf{2020 Mathematics Subject Classification~~~}60G51, 60H15
		
		\section{Introduction}\label{introduction}
		 In this paper, we consider the small mass limit $\ep\to0$ of the following stochastic wave equation driven by a cylindrical stable L\'{e}vy process 
		\begin{equation}
			\begin{cases}
				\ep\ube_{tt}+\ube_t=\Delta \ube+f(\ube)+\ep^\theta\dot{L},\\
				\ube(0)=u_0,\enspace\ube_t(0)=v_0,\enspace\ube|_{\pa D}=0.
			\end{cases}\label{wave}
		\end{equation}
		Here, $0\leq\theta<1$ is a constant, and $L$ is a cylindrical $\al$-stable process whose properties are detailed in Section~\ref{notionsandassumptions}.
		
		Formally, as $\ep\to 0$, the term $\ep\ube_{tt}$ in equations (\ref{wave}) vanishes, so that the approximation equation should be a stochastic heat equation, that is
		\begin{equation}\label{limitequation}
			\bube_t=\Delta\bar{\ube}+f(\bube)+\ep^\theta\dot{L},\enspace \bube(0)=u_0,\enspace \bube|_{\pa D}=0.
		\end{equation}
		Note that when $\theta=0,$ (\ref{limitequation}) does not depend on $\ep,$ and we write $\bar{U}$ instead of $\bube.$ The limit behavior of (\ref{wave}) as $\ep\to0$ in the case $\theta=0$ is known as Smoluchowski--Kramers (SK) approximation, which is proposed by Smoluchowski\cite{SMOL16} and Kramers\cite{KRAM40}. The SK approximation with cylindrical Wiener process is researched by Cerrai and Freidlin \cite{CF06A,CF06B}.  For more research about this topic, we refer to \cite{SWW24,chaos24, HMV15, HOTT13}. 
		
		Most of the aforementioned research concern about the case when the noise is Gaussian. However, it is found that L\'evy noise plays important roles in many fields such as epidemics spreading \cite{epidemics} and financial modeling \cite{finance}. $\al$-stable L\'evy process is an important class of L\'evy process, which has nice scaling property and self-similarity \cite{SATO99,KALL97}. We refer to \cite{SATO99, TAQQU} for more details about finite dimensional stable process. There are also fruitful works about cylindrical $\al$-stable process emerging in recent years\cite{AR10, BOD23, KOS20, KR22, JR17, PZ11}. In particular, Jakubowski and Riedle\cite{JR17} define the stochastic integral with respect to cylindrical L\'evy process. Kosmala and Riedle \cite{KR22} study the mild solution of SPDE driven by cylindrical stable process. 
		
		However, it seems that there is no work concerned about singular perturbation of wave equation with cylindrical stable process. Zhang \cite{ZHAN08} considers the situation when the noise is a finite dimensional L\'evy process with finite L\'evy measure, and the author himself considers the finite dimensional $\al$-stable case in \cite{ZHAO24}. There is remarkable technical difficulty when the noise is cylindrical stable noise. For example, the c\`adl\`ag property of the solution does not hold under classical hypothesis, see e.g. \cite{PZ10}. Besides, the second-order moment and the L\'evy measure of stable process are both infinite. Moreover, since the noise is non-Gaussian, the Garcia-Rademich-Rumsey theorem is invalid. It is also worth mentioning that the assumption (A1) of \cite[Page 295]{KR22} is not satisfied for our equations (\ref{wave}), so we have to give another definition of solution, which is analogous to the martingale solution in the Brownian case \cite[Definition 2.1]{LLB}.
		
		Rewrite the equation (\ref{wave}) as
		\begin{equation}
			\begin{cases}
				\ube_{t}=\vbe,\\
				\vbe_t=\ep^{-1}[-\vbe+\Delta\ube+f(\ube)]+\ep^{\theta-1}\dot{L},\\
				\ube(0)=u_0,\enspace  \vbe(0)=v_0,\enspace \ube|_{\pa D}=0.
			\end{cases}\label{slowfast}
		\end{equation}
		Equations (\ref{slowfast}) has a form of slow-fast system \cite{DW14}, where $\vbe$ can be viewed as a fast component, and $\ube$ can be seen as a slow component. Inspired by \cite{LRW11}, we split the fast component as
		\begin{equation}
			\begin{cases}
				\bvbe_{1,t}(t)=-\ep^{-1}\bvbe_1(t), \\
				\bvbe_{2,t}(t)=-\ep^{-1}[\bvbe_2(t)-\Delta\ube(t)-f(\ube(t))],\\
				\bvbe_{3,t}(t)=-\ep^{-1}\bvbe_3(t)+\ep^{-\frac{1}{\al}}L(t),\\
				\bvbe_1(0)=\ep v_0, \enspace \bvbe_2(0)=0, \enspace \bvbe_3(0)=0.
			\end{cases}\label{bveq}
		\end{equation}
It is straightforward to check that
		\begin{equation}
			\vbe=\ep^{-1}\bvbe_1+\bvbe_2+\ep^{\theta+\frac{1}{\al}-1}\bvbe_3. \label{decomposition}
		\end{equation}
		The splitting makes the analysis to (\ref{slowfast}) much more clear.
		
		The paper is organized as follows. In Section \ref{notionsandassumptions}, we introduce some notions and notations as well as impose some assumptions, after which we state the main theorem. In Section \ref{wellposed}, we give definition of the stochastic heat equation and the stochastic wave equation, and then establish the well-posedness of them, whose proof is postponed in Appendix \ref{appen}. In Section \ref{momentestimate}, we prove the uniform boundedness of $\ube$, which is fundamental to establish the tightness in Section \ref{tightnessofu}. With these preparation, we pass the limit and obtain the approximation equation in Section \ref{proofofthemainresult}.
		\section{Notions and Assumptions}\label{notionsandassumptions}
		Let $D$ be a bounded open subset of $\mathbb{R}^3$ with sufficiently smooth boundary. Denote $(\cdot,\cdot)$ and $||\cdot||$ the standard inner product and the norm on the Lebesgue space $H:=L^2(D)$, respectively. Define the operator $A$ by $Au:=-\Delta u$, where $u\in Dom(A):=\{u\in L^2(D), \Delta u\in L^2(D), u|_{\pa D}=0\}.$ Denote $(\al_i)_{i\in\mathbb{N}}$ and $(e_i)_{i\in\mathbb{N}}$ the eigenvalues and corresponding eigenvectors of $A$, satisfying that $0<\al_1\leq\al_2\leq...\uparrow\infty.$ For each $s>0,$ define $H^s_0(D)$ to be the closure of $C_0^\infty(D)$ (the set consists of all infinitely differentiable functions on $D$) with respect to the norm $$||u||_s^2:=\sum\limits_{k\in\mathbb{N}}\al_k^s(u,e_k)^2,$$ 
		and we denote $H^{-s}(D)$ as the dual space of $H_0^s(D)$ with dual norm $||\cdot||_{-s}.$
		
		Let us recall some concepts concerned about L\'{e}vy process and $\al$-stable law, which can be found in \cite[Chapter 4]{peszat2007stochastic}. If $L$ is a L\'{e}vy process taking values in a separable Hilbert space $H$, then there exists a real valued function $\psi$ defined on $H$, called the L\'{e}vy exponent of $L$, such that for all $h\in H,$
		$$\mathbb{E}e^{i(L(t),h)}=e^{-t\psi(h)}.$$ A real symmetric $\al$-stable random variable $X$ is a random variable with characteristic function$$\mathbb{E}e^{ihX}=e^{-\sigma^\al|h|^\al},\enspace h\in\mathbb{R},$$ where $\sigma>0$ is a constant. When $\sigma=1,$ it is called a normalized $\al$-stable random variable. A real normalized $\al$-stable L\'{e}vy process $L$ is a L\'{e}vy process such that $L(1)$ is a real normalized $\al$-stable random variable. It is obvious that in this case, the L\'{e}vy exponent of $L$ is $$\psi(h)=|h|^\al,\enspace h\in\mathbb{R}.$$
		
		Now we list some basic concepts of the Skorokhod space $\mathbb{D}$. Let $V$ be a separable Banach space. The Skorokhod space $\mathbb{D}([0,T];V)$ consists of all $V$-valued functions on $[0,T].$ For $x,y\in\mathbb{D}([0,T];V)$, set $$d^o(x,y):=\inf\limits_{\lambda\in\Lambda}\max\{||\lambda||^o, ||x-y\circ \lambda||\},$$
		where $||\lambda||^o:=\sup\limits_{0\leq s\leq t}\Big|\log\frac{\lambda(t)-\lambda(s)}{t-s}\Big|,$ and $\Lambda$ consists of all strictly increasing continuous bijection on $[0,T].$ It is well known that $d^o$ defines a complete separable metric on $\mathbb{D}([0,T];V)$. For more properties of Skorokhod space, we refer to \cite[Chapter 3]{BIL13}.
		
		Let $L$ be a normalized cylindrical $\al$-stable process on $(\Omega,\mathcal{F}, (\mathcal{F}_t)_{0\leq t\leq T}, \mathbb{P})$ for some $\al\in(1,2)$ with 
		\begin{equation}\label{cylinderical}
		L(t)=\sum\limits_{k=1}^{\infty}\lambda_k L_k(t)e_k,
		\end{equation}
		where $\lambda_k>0$ for each $k\in\mathbb{N},$ and  $(L_k)_{k\in\mathbb{N}}$ is a sequence of independent real normalized $\al$-stable L\'evy process.
		
		We make the following assumptions.
		
		\noindent($\mathbf{A_1}$) For some $\gamma>0$ with $\sum\limits_{k\in\mathbb{N}}\al_k^{1/2}\frac{1}{k^{1+\gamma}}<\infty$, 
		$$\sum\limits_{k\in\mathbb{N}}(\al_k^{3/2}\lambda_k+\al_k^{1/2}k^{1+\gamma}\lambda_k^2)<\infty.$$
		
		\noindent($\mathbf{A_2}$) $f:H\to H$ is globally Lipschitz, that is, there exists a constant $L$ such that for all $x,y\in H$, $$||f(x)-f(y)||\leq L||x-y||.$$
		
		\noindent($\mathbf{A_3}$) $f:H\to H$ satisfies the following growth condition: there exist a constant $C$ and $0<\delta<\al,$ such that for all $x\in H$,
		$$||f(x)||\leq C(1+||x||^{\delta/2}).$$
		
		\noindent($\mathbf{A_4}$) $u_0\in H_0^1(D),$ $v_0\in L^2(D).$ 
		
		\begin{remark}
			(i) It is well known that $\al_k=\mathcal{O}(k^{2/d}).$ Therefore, when $d=3$, $$\sum\limits_{k\in\mathbb{N}}\al_k^{1/2}\frac{1}{k^{1+\gamma}}<\infty$$ for each $\gamma>\frac{1}{3}.$
			
			(ii) By ($\mathbf{A_1}$), the series in (\ref{cylinderical}) converges almost surely.
		\end{remark}
		In this paper, we use the notation $x\lesssim y$ to indicate that there exists a constant $C$ whose value may change from line to line such that $x\leq Cy.$ Unless otherwise stated, $C$ never depends on $\ep$ nor the component index $k$. For a real number $x,$ $x^+:=\max\{x,0\}$ and $x^-:=\max\{-x,0\}$. Let $\xi$ and $\eta$ be two random elements, defined not necessarily on the same probability space. We use $\mathcal{L}(\xi)$ to denote the distribution of $\xi$, and write $\xi\overset{d}{=}\eta$ to imply that they have the same distribution. For two stochastic processes $X$ and $Y$ on the same probability basis, we say that $X$ and $Y$ are indistinguishable if there is a negligible set (i.e., a measurable set with zero measure), outside of which $X(t)=Y(t)$ for all $t$. 		
		Our main result is the following theorem.
		
		\begin{thm}\label{mainresult}
		(i) When $\theta=0$,
		$$\lim\limits_{\ep\to0}d^o(\ube,\bar{U})=0\quad\text{in probability}.$$
		
		(ii) When $\theta>0$, there exists a constant $C$ such that for each $\ep>0,$
		$$\mathbb{E}\supt||\ube(t)-\bube(t)||_{-2}\leq C\ep^\theta.$$
		\end{thm}
		\section{Well-posedness of SPDEs}\label{wellposed}
		We give definition of stochastic wave equations and stochastic heat equations driven by cylindrical stable process. 
		Let $L$ be a cylindrical stable process taking form of (\ref{cylinderical}). 
		\begin{definition}\label{defofsolution}
			A weak solution of a stochastic heat equation 
			\begin{equation}\label{formofSHE}
				u_t=\Delta u+f(u)+\dot{L},\enspace u(0)=u_0\in L^2(D),\enspace u|_{\pa D}=0
			\end{equation}
			is a $L^2(D)$-valued predictable process such that for each $\phi\in C^2([0,T]\times D)$ with $\phi|_{\partial D}=0$, the following equality holds in the sense that they are indistinguishable:
			\begin{eqnarray}\label{defofSHE}
				&&(u(t),\phi(t))-(u_0,\phi(0))\nonumber\\&=&\intot(u(s),\phi_t(s)+\Delta\phi(s))ds+\intot (f(u(s)),\phi(s))ds\nonumber\\&&+\sum_{k=1}^{\infty}\lambda_k\intot(\phi(s),e_k)dL_k(s).
			\end{eqnarray}
		\end{definition}
		Similarly, we give definition of the weak solution of a stochastic wave equation.
		
		\begin{definition}\label{defofsolution2}
			A weak solution of a stochastic wave equation
			\begin{equation}\label{formofSWE}
				\ep u_{tt}+u_t=\Delta u+f(u)+\dot{L},\enspace u(0)=u_0\in H_0^1(D),\enspace v(0)=v_0\in L^2(D), u|_{\pa D}=0,
			\end{equation}
			is a $H_0^1(D)\times L^2(D)$-valued predictable process $(u,v)$ such that for each $\phi\in C^1([0,T]\times D)$ with $\phi|_{\partial D}=0$, the following equalities hold in the sense that they are indistinguishable:
			\begin{eqnarray}\label{defofSWE}
				\begin{cases}
					(u(t),\phi(t))-(u_0,\phi(0))=\intot(u(s),\phi_t(s))+(v(s),\phi(s))ds,\\
					(v(t),\phi(t))-(v_0,\phi(0))\\=\frac{1}{\ep}\Big[\intot-(v(s),\phi(s))-(\nabla u(s),\nabla\phi(s))+(f(u(s)),\phi(s)))ds\Big]+\intot(v(s),\phi_t(s))ds\\+\frac{1}{\ep}\sum\limits_{k=1}^{\infty}\lambda_k\intot(\phi(s),e_k)dL_k(s).
				\end{cases}
			\end{eqnarray}
		\end{definition}
		\begin{remark}
			In Definition \ref{defofsolution} and Definition \ref{defofsolution2}, the term "weak" is in the sense of PDE. Our solutions are strong solutions in the sense of probability, since the stochastic basis is given and we never change it.
		\end{remark}
		\begin{prop}\label{wellposedness}
			(i) Equation (\ref{wave}) admits a weak solution $(\ube,\vbe)$ with trajectory belongs to $\mathbb{C}([0,T];H_0^1(D))\times\mathbb{D}([0,T]\times L^2(D))$ with probability 1. The solution is unique in the sense that if $(u_1,v_1)$ and $(u_2,v_2)$ are two weak solutions of (\ref{wave}), then they are indistinguishable.
			
			(ii) Equation (\ref{limitequation}) admits a  weak solution $\bube$ with trajectory belongs to $\mathbb{D}([0,T];L^2(D))$ with probability 1. The solution is unique in the sense that if $u_1$ and $u_2$ are two weak solutions of (\ref{limitequation}), then they are indistinguishable.
		\end{prop}
		The proof is postponed to Appendix \ref{appen}.
	
        \section{Moment estimate for $\ube$}\label{momentestimate}
       In this section, we give moment estimate for $\ube$. For this purpose, we firstly deal with $\ue$, which is the linear part of the stochastic wave equation, that is,
       \begin{equation}
       	\begin{cases}
       		\ep\ue_{tt}+\ue_t=\Delta \ue+\ep^\theta\dot{L},\\
       		\ue(0)=0,\enspace\ue_t(0)=0.
       	\end{cases}
       \end{equation} 
       In order to establish the moment estimate for $\ue$ and $\ube$, it suffices to consider the case $\theta=0,$ since the noise behaves less singularity when $\theta>0.$ We write this equation componentwise, that is for each $k\in\mathbb{N},$
       \begin{equation}\label{1Dwave}
       	\begin{cases}
       		\dot{u}^\ep_k(t)=\vek(t), \enspace\uek(0)=0,\\
       		\dot{v}^\ep_k(t)=-\frac{1}{\ep}[\al_k\uek(t)+\vek(t)]+\lambda_k \dot{L}_k(t),\enspace\vek(0)=0.
       	\end{cases}
       \end{equation}
       We begin with a lemma solving analytically a second-order linear ODE, whose proof is a direct computation.
       \begin{lemma}\label{anODEresult}(\cite[Proposition 2.2]{CF06A})
       	For each $\ep>0$ and $k\in\mathbb{N},$ let $(\fek,g^\ep_k)$ be the solution of
       	\begin{equation}
       		\begin{cases}
       			f^{\prime}(t)=g(t),\enspace f(0)=0,\\
       			\ep g^{\prime}(t)=-\al_k f(t)-g(t),\enspace g(0)=1.
       		\end{cases}
       	\end{equation}
       	Then
       	$$\fek(t)=\frac{1}{2}e^{-\frac{t}{2\ep}}\frac{1}{\gek}(e^{\gek t}-e^{-\gek t}),$$
       	$$g^\ep_k(t)=\frac{1}{2}e^{-\frac{t}{2\ep}}\Big[(1-\frac{1}{2\ep\gek})e^{\gek t}+(1+\frac{1}{2\ep\gek})e^{-\gek t}\Big],$$ where $\gek:=\frac{\sqrt{1-4\al_k\ep}}{2\ep}\in\mathbb{C}$ and $\frac{1}{\gek}(e^{\gek t}-e^{-\gek t}):=2t$ in the case $\gek=0.$ 	
       	
       \end{lemma}
     
       In the lemma below, we give an explicit expression for $\uek,$ which is a direct consequence of the variation-of-constant formula.
       \begin{lemma}\label{expessionforu}
       	For each $\ep>0$ and $k\in\mathbb{N},$ 
       	$$\uek(t)=\frac{\lambda_k}{\ep}\intot\fek(t-s)dL_k(s),$$
       	$$\vek(t)=\frac{\lambda_k}{\ep}\intot g^\ep_k(t-s)dL_k(s).$$
       \end{lemma}
      
		\noindent With the help of Lemma \ref{anODEresult} and Lemma \ref{expessionforu}, we give an important moment estimate for $\uek.$
		\begin{lemma}\label{H1norm}
			For each $1\leq p<\al,$
			$$\supe\supt \mathbb{E}|\uek(t)|^p\lesssim\frac{\lambda_k^p}{\al_k^{p/2}}.$$
		\end{lemma}
		\begin{proof}
			 Since $\uek(t)=\lambda_k\intot\frac{\fek(t-s)}{\ep}dL_k(s),$ by \cite[Theorem 3.2]{RW06} and a change-of-variable, we have
			\begin{eqnarray}\label{thirdlineusedonly}
				&&\mathbb{E}|\uek(t)|^p\nonumber\\
				&\lesssim&\lambda_k^p\Big(\mathbb{E}\intot\Big|\frac{\fek(t-s)}{\ep}\Big|^\al ds\Big)^{p/\al}\nonumber
				\\&=&\frac{\lambda_k^p}{\ep^p}\Big(\intot|\fek(s)|^\al\Big)^{p/\al}\nonumber\\&=&\frac{\lambda_k^p}{\ep^p}\Big(\intot\frac{1}{2^\al}e^{-\frac{s\al}{2\ep}}\frac{1}{|\gek|^\al}|e^{\gek s}-e^{-\gek s}|^\al ds\Big)^{p/\al}\nonumber\\&\lesssim&\frac{\lambda_k^p}{\ep^p|\gek|^p}\Big(\intot e^{-\frac{s\al}{2\ep}}|e^{\gek s}(1-e^{-2\gek s})|^\al ds\Big)^{p/\al}\nonumber\\&=&\frac{\lambda_k^p}{\ep^p|\gek|^p}\Big(\int_0^{t/\ep} e^{-\frac{s\al}{2}}|e^{\gek\ep s}(1-e^{-2\gek\ep s})|^\al\ep ds\Big)^{p/\al}\nonumber\\&=& \frac{\lambda_k^p}{\ep^{p-p/\al}|\gek|^p}\Big(\int_0^{t/\ep} e^{-\frac{s\al}{2}}|e^{\gek\ep s}(1-e^{-2\gek\ep s})|^\al ds\Big)^{p/\al}.
			\end{eqnarray}
			Recall that $\gek=\frac{\sqrt{1-4\al_k\ep}}{2\ep},$ so that we can continue our estimate as
			\begin{eqnarray}\label{generalestimateforuek}
				&&\mathbb{E}|\uek(t)|^p\nonumber\\&\lesssim&\frac{\lambda_k^p\ep^{p/\al}}{|\sqrt{1-4\al_k\ep}|^p}\Big(\int_0^{t/\ep} e^{-\frac{s\al}{2}}|e^{\gek\ep s}(1-e^{-2\gek\ep s})|^\al ds\Big)^{p/\al}\nonumber\\&=&\lambda_k^p\ep^{p/\al}\Big(\int_0^{t/\ep} e^{-\frac{s\al}{2}}|e^{\gek\ep s}|^\al\Big|\frac{1-e^{-2\gek\ep s}}{\sqrt{1-4\al_k\ep}}\Big|^\al ds\Big)^{p/\al}\nonumber\\&=&\lambda_k^p\ep^{p/\al}\Big(\int_0^{t/\ep} e^{-\frac{s\al}{2}}\Big|e^{\frac{\sqrt{1-4\al_k\ep}s}{2}}\Big|^\al\Big|\frac{1-e^{-2\gek\ep s}}{\sqrt{1-4\al_k\ep}}\Big|^\al ds\Big)^{p/\al}\nonumber\\&=&\lambda_k^p\ep^{p/\al}\Big(\int_0^{t/\ep} e^{-\frac{s\al}{2}}\Big|e^{\frac{\sqrt{(1-4\al_k\ep)^+}s}{2}}\Big|^\al\Big|\frac{1-e^{-2\gek\ep s}}{\sqrt{1-4\al_k\ep}}\Big|^\al ds\Big)^{p/\al}\nonumber\\&=&\lambda_k^p\ep^{p/\al}\Big(\int_0^{t/\ep} e^{-\frac{s\al}{2}+\frac{\sqrt{(1-4\al_k\ep)^+}s\al}{2}}\Big|\frac{1-e^{-2\gek\ep s}}{\sqrt{1-4\al_k\ep}}\Big|^\al ds\Big)^{p/\al}\nonumber\\&\leq&\lambda_k^p\ep^{p/\al}\Big(\int_0^{\infty} e^{-\frac{s\al}{2}+\frac{\sqrt{(1-4\al_k\ep)^+}s\al}{2}}\Big|\frac{1-e^{-2\gek\ep s}}{\sqrt{1-4\al_k\ep}}\Big|^\al ds\Big)^{p/\al}.
			\end{eqnarray}
\noindent\textbf{Case 1}\quad $0<\sqrt{1-4\al_k\ep}\leq\frac{1}{2}.$

\noindent Notice that \begin{eqnarray}
	&&e^{-\frac{s\al}{2}+\frac{\sqrt{(1-4\al_k\ep)^+}s\al}{2}}\Big|\frac{1-e^{-2\gek\ep s}}{\sqrt{1-4\al_k\ep}}\Big|^\al\nonumber\\&=&\Big(e^{-s+\sqrt{1-4\al_k\ep}s}\frac{|1-e^{-2\gek\ep s}|^2}{|1-4\al_k\ep|}\Big)^{\al/2}\nonumber\\&\leq&(e^{-s/2}s^2)^{\al/2}\nonumber\\&=&e^{-\al s/4}s^\al.
\end{eqnarray}
Substitute the inequality above into (\ref{generalestimateforuek}),
\begin{equation}\label{case1}
	\mathbb{E}|\uek(t)|^p\lesssim\lambda_k^p\ep^{\frac{p}{\al}}\Big(\int_{0}^{\infty}e^{-\al s/4}s^\al ds\Big)^{p/\al}\lesssim \lambda_k^p\ep^{\frac{p}{\al}}\lesssim \frac{\lambda_k^p}{\al_k^{p/\al}}.
\end{equation}
\noindent\textbf{Case 2}\quad $\sqrt{1-4\al_k\ep}>\frac{1}{2}.$

\noindent In this case, 
$$\Big|\frac{1-e^{-2\gek\ep s}}{\sqrt{1-4\al_k\ep}}\Big|^\al=\Big(\frac{|1-e^{-2\gek\ep s}|^2}{|1-4\al_k\ep|}\Big)^{\al/2}\lesssim 1,$$
Substitute the inequality above into (\ref{generalestimateforuek}),
\begin{equation*}
\mathbb{E}|\uek(t)|^p\lesssim\lambda_k^p\ep^{p/\al}\Big(\int_{0}^{\infty}e^{-\frac{s\al}{2}+\frac{\sqrt{(1-4\al_k\ep)^+}s\al}{2}}ds\Big)^{p/\al}.
\end{equation*}
But$$\int_{0}^{\infty}e^{-\frac{s\al}{2}+\frac{\sqrt{(1-4\al_k\ep)^+}s\al}{2}}ds=\frac{2(\sqrt{1-4\al_k\ep}+1)}{4\al_k\ep\al}\leq \frac{1}{\al_k\ep\al},$$
so that
\begin{equation}\label{case2}
	\mathbb{E}|\uek(t)|^p\lesssim\lambda_k^p\ep^{p/\al}(\frac{1}{\al_k\ep\al})^{p/\al}\lesssim \frac{\lambda_k^p}{\al_k^{p/\al}}.
\end{equation}
\noindent\textbf{Case 3}\quad $1-4\al_k\ep<0.$

\noindent Note that
\begin{eqnarray}
		&&e^{-\frac{s\al}{2}+\frac{\sqrt{(1-4\al_k\ep)^+}s\al}{2}}\Big|\frac{1-e^{-2\gek\ep s}}{\sqrt{1-4\al_k\ep}}\Big|^\al\nonumber\\&=&\Big(e^{-s+\sqrt{(1-4\al_k\ep)^+}s}\frac{|1-e^{-2\gek\ep s}|^2}{|1-4\al_k\ep|}\Big)^{\al/2}\nonumber\\&=&\Big[2e^{-s}\frac{1-cos(\sqrt{(1-4\al_k\ep)^-}s)}{(1-4\al_k\ep)^-}\Big]^{\al/2}\nonumber\\&\leq&\Big[\frac{s^2 e^{-s}}{(1-4\al_k\ep)^- s^2\vee1}\Big]^{\al/2}.
\end{eqnarray}
Substitute the inequality above into (\ref{generalestimateforuek}), 
\begin{eqnarray}\label{case3}
	&&\mathbb{E}|\uek(t)|^p\nonumber\\&\lesssim&\lambda_k^p\ep^{p/\al}\Big(\int_{0}^{\infty}\Big[\frac{s^2 e^{-s}}{(1-4\al_k\ep)^- s^2\vee1}\Big]^{\al/2}ds\Big)^{p/\al}\nonumber\\&=&\frac{\lambda_k^p\ep^{p/\al}}{\al_k^{p/2}\ep^{p/2}}\Big(\int_{0}^{\infty}\Big[\frac{\al_k\ep s^2 e^{-s}}{(1-4\al_k\ep)^- s^2\vee1}\Big]^{\al/2}ds\Big)^{p/\al}\nonumber\\&\lesssim&\frac{\lambda_k^p\ep^{p/\al-p/2}}{\al_k^{p/2}}\Big(\int_{0}^{\infty}[(1+s^2)e^{-s}]^{\al/2}ds\Big)^{p/\al}\nonumber\\&\lesssim&\frac{\lambda_k^p\ep^{p/\al-p/2}}{\al_k^{p/2}}\nonumber\\&\leq& \frac{\lambda_k^p}{\al_k^{p/2}}.
\end{eqnarray}
\noindent\textbf{Case 4}\quad $1-4\al_k\ep=0.$

\noindent From (\ref{case1}), (\ref{case2}) and (\ref{case3}), we know that for $\ep$ satisfying $1-4\al_k\ep\neq0$,
\begin{equation}
\mathbb{E}|\uek(t)|^p\lesssim \frac{\lambda_k^p}{\al_k^{p/\al}}\vee\frac{\lambda_k^p}{\al_k^{p/2}}=\frac{\lambda_k^p}{\al_k^{p/2}},
\end{equation}			
In fact, by inspecting into the third line of (\ref{thirdlineusedonly}), we have that for $\ep$ satisfying $1-4\al_k\ep\neq0,$
\begin{equation}\label{case41}
	\frac{\lambda_k^p}{\ep^p}\Big(\intot|\fek(s)|^\al ds\Big)^{p/\al}\lesssim \frac{\lambda_k^p}{\al_k^{p/2}},
\end{equation}
which is equivalent to say that for all $\mu\in(0,1]$ with $\mu\neq\frac{1}{4\al_k},$
\begin{equation}
	\intot\Big|\frac{f^\mu_k(s)}{\mu}\Big|^\al ds\lesssim \frac{1}{\al_k^{\al/2}}.
\end{equation}
By Fatou Lemma and (\ref{case41}), for $\ep=\frac{1}{4\al_k},$
\begin{eqnarray}
	\lambda_k^p\Big(\intot\Big|\frac{\fek(s)}{\ep}\Big|^\al ds\Big)^{p/\al}\leq\lambda_k^p\lim\limits_{\mu\rightarrow\ep,\mu\neq\ep}\Big(\intot\Big|\frac{\fek(s)}{\ep}\Big|^\al ds\Big)^{p/\al}\lesssim \frac{\lambda_k^p}{\al_k^{\al/2}}\leq\frac{\lambda_k^p}{\al_k^{p/2}}.
\end{eqnarray}
By inspecting into the third line of (\ref{thirdlineusedonly}) again, we conclude that
\begin{equation}\label{case4}
	\mathbb{E}|\uek(t)|^p\lesssim\frac{\lambda_k^p}{\ep^p}\Big(\intot|\fek(s)|^\al\Big)^{p/\al}\lesssim \frac{\lambda_k^p}{\al_k^{p/2}}.
\end{equation}

Combining the four cases above, we conclude that for each $0<\ep\leq 1$ and $k\in\mathbb{N},$ 
\begin{equation}\label{momentofu}
	\mathbb{E}|\uek(t)|^p\lesssim \frac{\lambda_k^p}{\al_k^{p/2}}.
\end{equation}

		\end{proof}
	
		Applying the lemma above, we give a moment estimate of $\uek$ which is uniform with respect to $t$.
		\begin{lemma}\label{H-1norm}
		  $$\supe\mathbb{E}\supt|\uek(t)|\lesssim \al_k^{1/2}\lambda_k+\lambda_k+\frac{1}{k^{1+\gamma}}+k^{1+\gamma}\lambda_k^2.$$
		\end{lemma}
		\begin{proof}
			We prove the lemma by a splitting technique, which is a finite dimensional version of (\ref{bveq}). Consider the following equations
			\begin{equation}\label{compodecompo}
				\begin{cases}
					\dot{w}^{\ep,k}_1(t)=-\ep^{-1}[\wek_1(t)+\al_k\uek(t)], \enspace\wek_1(0)=0,\\
					\dot{w}^{\ep,k}_2(t)=-\ep^{-1}\wek_2(t)+\ep^{-\frac{1}{\al}}\lambda_kL_k(t), \enspace\wek_2(0)=0.
				\end{cases}
			\end{equation}
			By a direct computation,
			\begin{equation}\label{compodecompo2}
				\dot{u}^\ep_k=\vek=\wek_1+\ep^{\theta+\frac{1}{\al}-1}\wek_2.
			\end{equation}
			From (\ref{compodecompo2}),
			$$\uek(t)=\intot\wek_1(s)ds+\ep^{\theta+\frac{1}{\al}-1}\intot\wek_2(s)ds.$$
			so$$\uek(t)\leq\int_0^{T}|\wek_1(s)|ds+\ep^{\theta+\frac{1}{\al}-1}\supt\Big|\intot\wek_2(s)ds\Big|,$$
			for all $0\leq t\leq T.$ Taking expectation,
			\begin{equation}\label{generalestimateforsupu}
				\mathbb{E}\supt|\uek(t)|\leq\int_{0}^{T}\mathbb{E}|\wek_1(t)|dt+\ep^{\theta+\frac{1}{\al}-1}\mathbb{E}\supt\Big|\intot\wek_2(t)ds\Big|.
			\end{equation}
			By (\ref{compodecompo}),$$\wek_1(t)=-\ep^{-1}e^{-\ep^{-1}t}\intot\al_ke^{\ep^{-1}s}\uek(s)ds.$$
			Taking expectation and applying Lemma \ref{H1norm},
			\begin{eqnarray}\label{esforw1}
				&&\mathbb{E}|\wek_1(t)|\nonumber\\&\leq&\ep^{-1}e^{-\ep^{-1}t}\al_k\intot e^{\ep^{-1}s}\mathbb{E}|\uek(s)|ds\nonumber\\&\leq&\ep^{-1}e^{-\ep^{-1}t}\al_k\intot e^{\ep^{-1}s}\frac{\lambda_k}{\al_k^{1/2}}ds\nonumber\\&=&\ep^{-1}e^{-\ep^{-1}t}\al_k^{1/2}\lambda_k\intot e^{\ep^{-1}s}ds\nonumber\\&\leq&\al_k^{1/2}\lambda_k.
			\end{eqnarray}
			By (\ref{compodecompo}), $$\wek_2(t)=-\ep^{-1}\intot\wek_2(s)ds+\ep^{-\frac{1}{\al}}\lambda_kL_k(t),$$
			we have
			\begin{equation}\label{generalestimateforsecondterm}
				\ep^{\theta+\frac{1}{\al}-1}\intot\wek_2(s)ds=\ep^{\theta}[\lambda_kL_k(t)-\ep^{\frac{1}{\al}}\wek_2(t)].
			\end{equation} 
			Let $\wek_3:=\ep^{\frac{1}{\al}}\wek_2$. By (\ref{compodecompo}),
			$$\dot{w}^{\ep,k}_3=-\ep^{-1}\wek_3+\lambda_k\dot{L_k},\enspace \wek_3(0)=0.$$
			Set $a(k):=\frac{1}{k^{2+2\gamma}}$ and $F_k(x):=(x^2+a(k))^{1/2}.$ It is straightforward that
			\begin{equation}\label{precisedotF}
				F^\prime_k(x)=\frac{x}{(x^2+a(k))^{1/2}},\quad F^{\prime\prime}_k(x)=\frac{a(k)}{(x^2+a(k))^{3/2}},
			\end{equation}
			so that
			\begin{equation}\label{propertyofF}
				|F^\prime_k(x)|\leq 1,\quad |F^{\prime\prime}_k(x)|\leq \frac{1}{a(k)^\frac{1}{2}}.
			\end{equation}
			By It\^o formula,
			\begin{eqnarray}\label{generalestimateforrho}
				&&F_k(\wek_3(t))-a(k)^{1/2}\nonumber\\&=&-\ep^{-1}\intot F^\prime_k(\wek_3(s))\wek_3(s)ds\nonumber\\&&+\intotxl F_k(\wek_3(s-)+\lambda_kx)-F_k(\wek_3(s-))\tilde{N}_k(dsdx)\nonumber\\&&+\intotxg F_k(\wek_3(s-)+\lambda_kx)-F_k(\wek_3(s-))N_k(dxdx)\nonumber\\&&+\intotxl F_k(\wek_3(s)+\lambda_kx)-F_k(\wek_3(s))-F^\prime_k(\wek_3(s))\nu_k(dx)ds\nonumber\\&&=:\sum\limits_{i=1}^{4}I_i(t).
			\end{eqnarray}
		By (\ref{precisedotF}),
		\begin{equation}\label{esforI1}
			\mathbb{E}\supt I_1(t)=\mathbb{E}\supt\Big(-\ep^{-1}\intot\frac{|\wek_3(s)|^2}{(|\wek_3(s)|^2+a(k))^{1/2}}ds\Big)\leq 0.
		\end{equation}
	By Burkholder-Davis-Gundy inequality, Jensen inequality, Taylor formula and (\ref{propertyofF})
	\begin{eqnarray}\label{esforI2}
		&&\mathbb{E}\supt|I_2(t)|\nonumber\\&\leq&\mathbb{E}\Big[\intobtxl|F_k(\wek_3(s-)+\lambda_kx)-F_k(\wek_3(s-))|^2N_k(dsdx)\Big]^{1/2}\nonumber\\&\leq&\Big[\mathbb{E}\intobtxl|F_k(\wek_3(s-)+\lambda_kx)-F_k(\wek_3(s-))|^2N_k(dsdx)\Big]^{1/2}\nonumber\\&=&\Big[\mathbb{E}\intobtxl|F_k(\wek_3(s-)+\lambda_kx)-F_k(\wek_3(s-))|^2\nu_k(dx)ds\Big]^{1/2}\nonumber\\&=&\Big[\mathbb{E}\intobtxl|F^\prime_k(\wek_3(s-)+\theta\lambda_kx)\lambda_kx|^2\nu_k(dx)ds\Big]^{1/2}\nonumber\\&\leq&\Big[\mathbb{E}\intobtxl\lambda_k^2x^2\nu_k(dx)ds\Big]^{1/2}\nonumber\\&=&\lambda_k\Big(\int_{|x|<1}x^2\nu_1(dx)\Big)^{1/2}T^{1/2},
	\end{eqnarray}
	\begin{eqnarray}\label{esforI3}
		&&\mathbb{E}\supt|I_3(t)|\nonumber\\&\leq&\mathbb{E}\intobtxg|F_k(\wek_3(s-)+\lambda_kx)-F_k(\wek_3(s-))|N_k(dsdx)\nonumber\\&=&\mathbb{E}\intobtxg|F_k(\wek_3(s-)+\lambda_kx)-F_k(\wek_3(s-))|\nu_k(dx)ds\nonumber\\&\leq&\mathbb{E}\intobtxg|F^\prime_k(\wek_3(s-)+\theta\lambda_kx)\lambda_kx|\nu_k(dx)ds\nonumber\\&\leq&\mathbb{E}\intobtxg\lambda_k|x|\nu_k(dx)ds\nonumber\\&=&\lambda_k\int_{|x|\geq 1}|x|\nu_1(dx)T,
		\end{eqnarray}
		and
		\begin{eqnarray}\label{esforI4}
			&&\mathbb{E}\supt|I_4(t)|\nonumber\\&\leq&\mathbb{E}\intobtxl|F_k(\wek_3(s)+\lambda_kx)-F_k(\wek_3(s))-F^\prime_k(\wek_3(s))\lambda_kx|\nu_k(dx)ds\nonumber\\&=&\mathbb{E}\intobtxl F^{\prime\prime}_k(\wek_3(s)+\theta\lambda_kx)\lambda_k^2x^2\nu_k(dx)ds\nonumber\\&\leq&\frac{1}{a(k)^{1/2}}\lambda_k^2\int_{|x|<1}x^2\nu_1(dx)T.
		\end{eqnarray}
		Substituting (\ref{esforI1})-(\ref{esforI4}) into (\ref{generalestimateforrho}),
		\begin{equation*}
			\mathbb{E}\supt F_k(\wek_3(t))\lesssim a(k)^{1/2}+\lambda_k+\frac{\lambda_k^2}{a(k)^{1/2}},
		\end{equation*}
		which means
		\begin{equation}\label{esforrho}
			\mathbb{E}\supt|\ep^{\frac{1}{\al}}\wek_2(t)|\lesssim a(k)^{1/2}+\lambda_k+\frac{\lambda_k^2}{a(k)^{1/2}}.
		\end{equation}
		Next we show that
		\begin{equation}\label{esforl}
		\mathbb{E}\supt|L_k(t)|\lesssim1,
		\end{equation}
		whose proof is a standard argument. In fact, by Burkholder-Davis-Gundy inequality,
			\begin{eqnarray}
				&&\mathbb{E}\supt|L_k(t)|\nonumber\\&\leq&\mathbb{E}\supt\Big|\intotxl x\tilde{N}_k(dsdx)\Big|+\mathbb{E}\supt\Big|\intotxg x{N}_k(dsdx)\Big|\nonumber\\&\leq&\mathbb{E}\sqrt{\int_{0}^{T}\int_{|x|<1}x^2 \tilde{N}_k(dsdx)}+\mathbb{E}\int_{0}^{T}\int_{|x|\geq1}|x|N_k(dsdx)\nonumber\\&\leq&\sqrt{\mathbb{E}\int_{0}^{T}\int_{|x|<1}x^2 \tilde{N}_k(dsdx)}+\mathbb{E}\int_{0}^{T}\int_{|x|\geq1}|x|N_k(dsdx)\nonumber\\&=&\sqrt{T\int_{|x|<1}x^2\nu_1(dx)}+T\int_{|x|\geq1}|x|\nu_1(dx)\nonumber\\&\lesssim&1.
			\end{eqnarray}
	By (\ref{generalestimateforsecondterm}), (\ref{esforrho}) and (\ref{esforl}),
	\begin{equation}\label{esforsecondterm}
		\ep^{\theta+\frac{1}{\al}-1}\mathbb{E}\supt\Big|\intot\wek_2(s)ds\Big| \lesssim a(k)^{1/2}+\lambda_k+\frac{\lambda_k^2}{a(k)^{1/2}}.
	\end{equation}
		The lemma follows from combining estimates (\ref{generalestimateforsupu}), (\ref{esforw1}) and (\ref{esforsecondterm}).
			
		\end{proof}
Now we are ready to give the desired moment estimate for $\ube$.  
\begin{prop}\label{Hnormforub}
	\begin{equation}
		\supe\mathbb{E}\supt||\nabla\ube(t)||\lesssim1.
	\end{equation}
\end{prop}
		\begin{proof}
		Set $\re:=\ube-\ue.$ By Lemma \ref{H-1norm} and ($\mathbf{A_1}$),
		\begin{equation}\label{finitemomentofu}
			\supe\mathbb{E}\supt||\nabla \ue(t)||\lesssim1,
		\end{equation}  so it suffices to prove that  
		\begin{eqnarray}\label{esforre}
			\supe\mathbb{E}\supt||\nabla\re(t)||\lesssim1.
		\end{eqnarray}
		By a direct computation,
		\begin{equation}\label{eqforre}
			\ep\re_{tt}+\re_t=\Delta\re+f(\re+\ue),\enspace\re(0)=u_0,\enspace\re_t(0)=v_0.
		\end{equation}
	Multiplying $\re_t$ on both sides,
	$$\ep(\re_{tt},\re_t)+||\re_t||^2=(\Delta \re,\re_t)+(f(\ue+\re),\re_t).$$	
	Integrating with respect to $t$ and rearranging,
	\begin{eqnarray*}
		&&\frac{1}{2}||\nabla\re(t)||^2+\frac{\ep}{2}||\re_t(t)||^2+\intot||\re_t(s)||^2ds\\&=&\frac{1}{2}||u_0||_1^2+\frac{\ep}{2}||v_0||^2+\intot(f(\ue(s)+\re(s)),\re_t(s)) ds.	
	\end{eqnarray*}
	Using Cauchy-Schwarz inequality,
	\begin{eqnarray*}
		&&\frac{1}{2}||\nabla\re(t)||^2+\frac{\ep}{2}||\re_t(t)||^2+\intot||\re_t(s)||^2ds\\&\leq&\frac{1}{2}||u_0||_1^2+\frac{\ep}{2}||v_0||^2+\frac{1}{2}\intot||f(\ue(s)+\re(s))||^2+||\re_t(s)||^2 ds.	
	\end{eqnarray*}
	Rearranging,
	\begin{eqnarray*}
		&&\frac{1}{2}||\nabla\re(t)||^2+\frac{\ep}{2}||\re_t(t)||^2+\frac{1}{2}\intot||\re_t(s)||^2ds\\&\leq&\frac{1}{2}||u_0||_1^2+\frac{\ep}{2}||v_0||^2+\frac{1}{2}\intot||f(\ue(s)+\re(s))||^2ds.	
	\end{eqnarray*}
	In particular,
	\begin{eqnarray*}
		||\nabla\re(t)||^2\lesssim 1+\intot||f(\ue(s)+\re(s))||^2ds.
	\end{eqnarray*}
	By ($\mathbf{A_3}$),
	\begin{eqnarray*}
		&&||\nabla\re(t)||^2\\&\lesssim&1+\intot(1+||\ue(s)+\re(s)||^{\delta/2})^2ds\\&\lesssim&1+\intot(1+||\ue(s)||^{\delta/2}+||\re(s)||^{\delta/2})^2ds\\&\lesssim&1+\intot||\ue(s)||^{\delta}ds+\intot||\re(s)||^{\delta}ds.
	\end{eqnarray*}
	Using elementary inequalities $|a+b+c|^{\delta/2}\leq a^{\delta/2}+b^{\delta/2}+c^{\delta/2}$ and $|a|^{\delta/2}\leq 1+|a|,$
	\begin{eqnarray*}
		&&||\nabla\re(t)||^\delta\\&\lesssim&1+\Big[\intot||\ue(s)||^{\delta}ds\Big]^{\delta/2}+\Big[\intot||\re(s)||^{\delta}ds\Big]^{\delta/2}\\&\lesssim&1+\intot||\ue(s)||^{\delta}ds+\intot||\re(s)||^{\delta}ds\\&\lesssim&1+\int_{0}^{T}||\ue(t)||^{\delta}dt+\int_{0}^{T}||\nabla\re(s)||^{\delta}ds\\&\lesssim&1+\int_{0}^{T}||\ue(t)||^{\delta}dt+\int_{0}^{T}\sup_{0\leq s\leq t}||\nabla\re(s)||^{\delta}dt.
	\end{eqnarray*}
Taking supremum and expectation,
\begin{eqnarray*}
	&&\mathbb{E}\supt||\nabla\re(t)||^\delta\\&\lesssim& 1+\int_{0}^{T}\mathbb{E}||\ue(t)||^{\delta}dt+\int_{0}^{T}\mathbb{E}\sup_{0\leq s\leq t}||\nabla\re(s)||^{\delta}dt.
\end{eqnarray*}
By Lemma \ref{H1norm} and ($\mathbf{A_1}$), 
$$\supe\supt\mathbb{E}||\ue(t)||^\delta<\infty,$$
so that 
\begin{eqnarray*}
	&&\mathbb{E}\supt||\nabla\re(t)||^\delta\\&\lesssim& 1+\int_{0}^{T}\mathbb{E}\sup_{0\leq s\leq t}||\nabla\re(s)||^{\delta}dt.
\end{eqnarray*}
The result follows from Gronwall inequality.
		\end{proof}
		
		\section{Tightness of $(\ube)_{0<\ep\leq 1}$}\label{tightnessofu}
		In this section, we establish the tightness of $(\ube)_{0<\ep\leq 1}$ in $\mathbb{D}([0,T];H^{-1}(D))$. 
		For this purpose, we check the condition in \cite[Theorem 7]{KR22}, which is an infinite dimensional version of the Aldous tightness criterion \cite[Theorem 16.10]{BIL13}. As in Section \ref{momentestimate}, we firstly deal with the linear part $\ue$, and then give estimate to the difference $\re$.
		\begin{lemma}\label{aldousforuek}
			For each $k\in\mathbb{N}$, $\delta\in (0,\infty)$ and stopping time $\tau$ with $\tau+\delta\leq T$,
			 $$\mathbb{E}|\uek(\tau+\delta)-\uek(\tau)|\lesssim\al_k(\al_k^{1/2}\lambda_k+\lambda_k+\frac{1}{k^{1+\gamma}}+k^{1+\gamma}\lambda_k^2)\delta+\lambda_k\sqrt{\delta}.$$
		\end{lemma}
		\begin{proof}
		Since $$d\vek(t)=\frac{1}{\ep}(-\vek(t)-\al_k\uek(t))dt+\ep^{\theta-1}\lambda_kdL_k(t),$$ we have
		$$\vek(t)=\intot -\frac{\al_k}{\ep}e^{-\ep^{-1}(t-s)}u_k(s)ds+\intot \ep^{\theta-1}\lambda_k e^{-\ep^{-1}(t-s)}dL_k(s),$$ so that \begin{eqnarray}
			&&\uek(t)\nonumber\\&=&\intot \vek(s)ds\\\nonumber&=&\intot\int_{0}^{s}-\frac{\al_k}{\ep}e^{-\ep^{-1}(s-r)}\uek(r)drds+\intot\int_{0}^{s}\ep^{\theta-1}\lambda_k e^{-\ep^{-1}(s-r)}dL_k(r)ds\nonumber\\&=:&I_k(t)+J_k(t).
		\end{eqnarray}
		By Fubini theorem,
		\begin{eqnarray}\label{fubiniargument}
			&&I_k(t)\nonumber\\&=&-\frac{\al_k}{\ep}\intot\int_{r}^{t}e^{-\ep^{-1}(s-r)}\uek(r)dsdr\nonumber\\&=&-\frac{\al_k}{\ep}\intot e^{\ep^{-1}r}\uek(r)\int_{r}^{t}e^{-\ep^{-1}s}dsdr\nonumber\\&=&-\frac{\al_k}{\ep}\intot e^{\ep^{-1}r}\uek(r)(\ep e^{-\ep^{-1}r}-\ep e^{-\ep^{-1}t})dr\nonumber\\&=&-\al_k\intot\uek(r)dr+\al_k\intot e^{-\ep^{-1}(t-r)}\uek(r)dr.
		\end{eqnarray}
		Therefore,
		\begin{eqnarray*}
			&&\frac{1}{\al_k}|I_k(\tau+\delta)-I_k(\tau)|\\&\leq&\int_{\tau}^{\tau+\delta}|\uek(s)|ds+\Big|e^{-\ep^{-1}(\tau+\delta)}\int_{0}^{\tau+\delta}e^{\ep^{-1}s}\uek(s)ds-e^{-\ep^{-1}\tau}\int_{0}^{\tau}e^{\ep^{-1}s}\uek(s)ds\Big|\nonumber\\&\leq&\int_{\tau}^{\tau+\delta}|\uek(s)|ds+\Big|e^{-\ep^{-1}(\tau+\delta)}\int_{0}^{\tau+\delta}e^{\ep^{-1}s}\uek(s)ds-e^{-\ep^{-1}(\tau+\delta)}\int_{0}^{\tau}e^{\ep^{-1}s}\uek(s)ds\Big|\nonumber\\&&+\Big|e^{-\ep^{-1}(\tau+\delta)}\int_{0}^{\tau}e^{\ep^{-1}s}\uek(s)ds-e^{-\ep^{-1}\tau}\int_{0}^{\tau}e^{\ep^{-1}s}\uek(s)ds\Big|\nonumber\\&=&\int_{\tau}^{\tau+\delta}|\uek(s)|ds+e^{-\ep^{-1}(\tau+\delta)}\Big|\int_{\tau}^{\tau+\delta}e^{\ep^{-1}s}\uek(s)ds\Big|\nonumber\\&&+e^{-\ep^{-1}\tau}(1-e^{-\ep^{-1}\delta})\Big|\int_{0}^{\tau}e^{\ep^{-1}s}\uek(s)ds\Big|\nonumber\\&\leq&\int_{\tau}^{\tau+\delta}|\uek(s)|ds+\int_{\tau}^{\tau+\delta}|\uek|ds+\ep^{-1}\delta e^{-\ep^{-1}\tau}\int_{0}^{\tau}e^{\ep^{-1}s}|\uek(s)|ds\nonumber\\&\leq&\supt|\uek(t)|\delta+\supt|\uek(t)|\delta+\supt|\uek(t)|\ep^{-1}\delta e^{-\ep^{-1}\tau}\int_{0}^{\tau}e^{\ep^{-1}s}ds\nonumber\\&\leq&3\supt|\uek(t)|\delta.
		\end{eqnarray*}
		Rearranging the inequality above and taking expectation,
		\begin{equation}
			\mathbb{E}|I_k(\tau+\delta)-I_k(\tau)|\lesssim\al_k\mathbb{E}\supt|\uek(t)|\delta.
		\end{equation}
		By Lemma \ref{H-1norm},
		\begin{equation}\label{tightIk}
			\mathbb{E}|I_k(\tau+\delta)-I_k(\tau)|\lesssim\al_k(\al_k^{1/2}\lambda_k+\lambda_k+\frac{1}{k^{1+\gamma}}+k^{1+\gamma}\lambda_k^2)\delta.
		\end{equation}

		Now we deal with $J_k.$ By the L\'{e}vy-It\^{o} decomposition,
		\begin{eqnarray}\label{tightJkready}
			&&J_k(t)\nonumber\\&=&\ep^{\theta-1}\lambda_k\intot\intos\int_{|x|\geq 1}e^{-\frac{1}{\ep}(s-r)}xN_k(drdx)ds\nonumber\\&&+\ep^{\theta-1}\lambda_k\intot\intos\int_{|x|< 1}e^{-\frac{1}{\ep}(s-r)}x\tilde{N}_k(drdx)ds\nonumber\\&=:&J^1_k(t)+J^2_k(t).
		\end{eqnarray}
		First we deal with $J^1_k.$ By integration-by-part formula for real semimartingale(\cite[page 68]{PRO05}), 
		\begin{eqnarray}\label{tightJ1kready}
			&&J^1_k(t)\nonumber\\&=&\ep^\theta\lambda_k\intot\ep^{-1}e^{-\ep^{-1}s}\intos\int_{|x|\geq 1}e^{\ep^{-1}r}xN_k(drdx)ds\nonumber\\&=&\ep^\theta\lambda_k\intot\intos\int_{|x|\geq 1}e^{\ep^{-1}r}xN_k(drdx)d(-e^{-\ep^{-1}s})\nonumber\\&=&\ep^\theta\lambda_k\Big[-\intot\int_{|x|\geq 1}e^{\ep^{-1}r}xN_k(drdx)e^{-\ep^{-1}t}+\intot e^{-\ep^{-1}s}d\intos\int_{|x|\geq 1}e^{\ep^{-1}r}xN_k(drdx)\Big]\nonumber\\&=&\ep^\theta\lambda_k\Big[-\intot\int_{|x|\geq 1}e^{\ep^{-1}r}xN_k(drdx)e^{-\ep^{-1}t}+\intot\int_{|x|\geq 1}xN_k(dsdx)\Big]\nonumber\\&=:&\ep^\theta\lambda_k(J^{11}_k(t)+J^{12}_k(t)).
		\end{eqnarray}
		For $J^{11}_k,$
		\begin{eqnarray}
		&&\mathbb{E}|J^{11}_k(\tau+\delta)-J^{11}_k(\tau)|\nonumber\\&=&\mathbb{E}|-e^{-\ep^{-1}(\tau+\delta)}\int_{0}^{\tau+\delta}\int_{|x|\geq 1}e^{\ep^{-1}r}xN_k(drdx)+e^{-\ep^{-1}\tau}\int_{0}^{\tau}\int_{|x|\geq 1}e^{\ep^{-1}r}xN_k(drdx)|\nonumber\\&\leq&\mathbb{E}|-e^{-\ep^{-1}(\tau+\delta)}\int_{0}^{\tau+\delta}\int_{|x|\geq 1}e^{\ep^{-1}r}xN_k(drdx)+e^{-\ep^{-1}(\tau+\delta)}\int_{0}^{\tau}\int_{|x|\geq 1}e^{\ep^{-1}r}xN_k(drdx)|\nonumber\\&&+\mathbb{E}|e^{-\ep^{-1}(\tau+\delta)}\int_{0}^{\tau}\int_{|x|\geq 1}e^{\ep^{-1}r}xN_k(drdx)-e^{-\ep^{-1}\tau}\int_{0}^{\tau}\int_{|x|\geq 1}e^{\ep^{-1}r}xN_k(drdx)|\nonumber\\&=:&A+B.
		\end{eqnarray}
		By a direct computation,
		\begin{eqnarray}
			&&A\nonumber\\&=&\mathbb{E}\Big|-e^{-\ep^{-1}(\tau+\delta)}\int_{\tau}^{\tau+\delta}\int_{|x|\geq 1}e^{\ep^{-1}r}xN_k(drdx)\Big|\nonumber\\&\leq&\mathbb{E}\Big[-e^{-\ep^{-1}(\tau+\delta)}\int_{\tau}^{\tau+\delta}\int_{|x|\geq 1}|e^{\ep^{-1}r}x|N_k(drdx)\Big]\nonumber\\&\leq&\mathbb{E}\Big[\int_{\tau}^{\tau+\delta}\int_{|x|\geq 1}|x|N_k(drdx)\Big]\nonumber\\&=&\mathbb{E}\Big[\int_{\tau}^{\tau+\delta}\int_{|x|\geq 1}|x|\nu_k(dx)dr\Big]\nonumber\\&=&\mathbb{E}\Big[\int_{0}^{T}\int_{|x|\geq 1}|x|\mathbb{I}_{[\tau,\tau+\delta]}(r)\nu_1(dx)dr\Big]\nonumber\\&=&\int_{0}^{T}\int_{|x|\geq 1}|x|\mathbb{E}\mathbb{I}_{[\tau,\tau+\delta]}(r)\nu_1(dx)dr\nonumber\\&=&\int_{|x|\geq 1}|x|\nu_1(dx)\int_{0}^{T}\mathbb{E}\mathbb{I}_{[\tau,\tau+\delta]}(r)dr\nonumber\\&=&\int_{|x|\geq 1}|x|\nu_1(dx)\delta\nonumber\\&\lesssim&\delta,
		\end{eqnarray}
		and
		\begin{eqnarray}
			&&B\nonumber\\&=&\mathbb{E}\Big|(e^{-\ep^{-1}\tau}-e^{-\ep^{-1}(\tau+\delta)})\int_{0}^{\tau}\int_{|x|\geq 1}e^{\ep^{-1}r}xN_k(drdx)\Big|\nonumber\\&=&\mathbb{E}\Big|e^{-\ep^{-1}\tau}(1-e^{-\ep^{-1}\delta})\int_{0}^{\tau}\int_{|x|\geq 1}e^{\ep^{-1}r}xN_k(drdx)\Big|\nonumber\\&\leq&\frac{\delta}{\ep}\mathbb{E}\Big|e^{-\ep^{-1}\tau}\int_{0}^{\tau}\int_{|x|\geq 1}e^{\ep^{-1}r}xN_k(drdx)\Big|\nonumber\\&\leq&\frac{\delta}{\ep}\mathbb{E}\Big[\int_{0}^{\tau}\int_{|x|\geq 1}e^{-\ep^{-1}(\tau-r)}|x|N_k(drdx)\Big]\nonumber\\&\leq&\frac{\delta}{\ep}\mathbb{E}\Big[\int_{0}^{\tau}\int_{|x|\geq 1}e^{-\ep^{-1}(\tau-r)}|x|\nu_1(dx)dr\Big]\nonumber\\&=&\frac{\delta}{\ep}\mathbb{E}\Big[\int_{0}^{T}\int_{|x|\geq 1}e^{-\ep^{-1}(\tau-r)}|x|\mathbb{I}_{[0,\tau]}(r)\nu_1(dx)dr\Big]\nonumber\\&=&\frac{\delta}{\ep}\int_{|x|\geq 1}|x|\nu_1(dx)\int_{0}^{T}\mathbb{E}(\mathbb{I}_{[0,\tau]}(r)e^{-\ep^{-1}(\tau-r)})dr\nonumber\\&=&\frac{\delta}{\ep}\int_{|x|\geq 1}|x|\nu_1(dx)\mathbb{E}(e^{-\ep^{-1}\tau}\int_{0}^{\tau}e^{\ep^{-1}r}dr)\nonumber\\&\lesssim&\delta.
		\end{eqnarray}
		By the two estimates above, we have
		\begin{equation}\label{tightJ11k}
			\mathbb{E}|J^{11}_k(\tau+\delta)-J^{11}_k(\tau)|\lesssim\delta.
		\end{equation}
		
		For $J^{12}_k, $by Fubini theorem,
		\begin{eqnarray}\label{tightJ12k}
			&&\mathbb{E}|J^{12}_k(\tau+\delta)-J^{12}_k(\tau)|\nonumber\\&=&\mathbb{E}\Big|\int_{\tau}^{\tau+\delta}\int_{|x|\geq 1}xN_k(dsdx)\Big|\nonumber\\&\leq&\mathbb{E}\int_{\tau}^{\tau+\delta}\int_{|x|\geq 1}|x|N_k(dsdx)\nonumber\\&=&\mathbb{E}\int_{\tau}^{\tau+\delta}\int_{|x|\geq 1}|x|\nu_k(dx)ds\nonumber\\&\leq&\mathbb{E}\int_{0}^{T}\int_{|x|\geq 1}\mathbb{I}_{[\tau,\tau+\delta]}(s)|x|\nu_1(dx)ds\nonumber\\&=&\int_{0}^{T}\int_{|x|\geq1}\mathbb{E}(\mathbb{I}_{[\tau,\tau+\delta]}(s)|x|)\nu_1(dxds)\nonumber\\&=&\int_{|x|\geq 1}|x|\nu_1(dx)\int_{0}^{T}\mathbb{E}\mathbb{I}_{[\tau,\tau+\delta]}(s)ds\nonumber\\&=&\int_{|x|\geq 1}|x|\nu_1(dx)\mathbb{E}\int_{0}^{T}\mathbb{I}_{[\tau,\tau+\delta]}(s)ds\nonumber\\&\lesssim&\delta.
		\end{eqnarray}
		From (\ref{tightJ11k}) and (\ref{tightJ12k}), we have
		\begin{equation}\label{tightJ1k}
			\mathbb{E}|J^1_k(\tau+\delta)-J^1_k(\tau)|\lesssim\ep^{\theta}\lambda_k\delta\leq\lambda_k\delta.
		\end{equation}
		We turn to deal with $J^2_k.$ For notation simplicity, we set $\tilde{L}_k(t):=\intot\int_{|x|< 1}x\tilde{N}_k(dsdx),$ so that $$J^2_k(t)=\lambda_k\intot\int_{0}^{u}\ep^{\theta-1}e^{-\ep^{-1}u}e^{\ep^{-1}h}d\tilde{L}_k(h)du,$$ and
		$$J^2_k(t)-J^2_k(s)=\lambda_k\int_{s}^{t}\int_{0}^{u}\ep^{\theta-1}e^{-\ep^{-1}u}e^{\ep^{-1}h}d\tilde{L}_k(h)du.$$
		By stochastic Fubini theorem(\cite[Theorem 64]{PRO05}),
		$$J^2_k(t)-J^2_k(s)=\lambda_k\Big[\intos\int_{s}^{t}\ep^{\theta-1}e^{-\ep^{-1}u}e^{\ep^{-1}h}dud\tilde{L}_k(h)+\int_{s}^{t}\int_{h}^{t}\ep^{\theta-1}e^{-\ep^{-1}u}e^{\ep^{-1}h}dud\tilde{L}_k(h)\Big].$$
		Therefore,
		\begin{eqnarray}\label{tightJ2kready}
 &&J^2_k(\tau+\delta)-J^2_k(\tau)\nonumber\\&=&\lambda_k\Big[\int_{0}^{\tau}\int_{\tau}^{\tau+\delta}\ep^{\theta-1}e^{-\ep^{-1}u}e^{\ep^{-1}h}dud\tilde{L}_k(h)+\int_{\tau}^{\tau+\delta}\int_{h}^{\tau+\delta}\ep^{\theta-1}e^{-\ep^{-1}u}e^{\ep^{-1}h}dud\tilde{L}_k(h)\Big]\nonumber\\&=:&\lambda_k (J^{21}_k+J^{22}_{k}).
		\end{eqnarray}
		By the definition of $\tilde{L}$, It\^{o} isometry, and H\"older inequality,
		\begin{eqnarray}\label{geforj2k}
			&&\mathbb{E}|J^{21}_k|^2\nonumber\\&=&\mathbb{E}\int_{0}^{\tau}\int_{|x|< 1}\Big|\int_{\tau}^{\tau+\delta}x\ep^{\theta-1}e^{-\ep^{-1}u}e^{\ep^{-1}h}du\Big|^2\nu_1(dx)dh\nonumber\\&=&\int_{|x|< 1}|x|^2\nu_1(dx)\mathbb{E}\int_{0}^{\tau}\Big|\int_{\tau}^{\tau+\delta}\ep^{\theta-1}e^{-\ep^{-1}u}e^{\ep^{-1}h}du\Big|^2dh\nonumber\\&\lesssim&\mathbb{E}\int_{0}^{\tau}\Big|\int_{\tau}^{\tau+\delta}\ep^{\theta-1}e^{-\ep^{-1}u}e^{\ep^{-1}h}du\Big|^2dh\nonumber\\&\leq&\mathbb{E}\int_{0}^{\tau}\Big(\int_{\tau}^{\tau+\delta}\ep^{2\theta-2}e^{-2\ep^{-1}u}e^{2\ep^{-1}h}du\Big)\delta dh\nonumber\\&=&\delta\ep^{2\theta-2}\mathbb{E}\int_{0}^{\tau}\int_{\tau}^{\tau+\delta}e^{-2\ep^{-1}(u-h)}dudh.
		\end{eqnarray}
		A direct calculation yields
		$$\int_{0}^{\tau}\int_{\tau}^{\tau+\delta}e^{-2\ep^{-1}(u-h)}dudh\lesssim\ep^2,$$ so that
		\begin{equation}\label{tightJ21k}
			\mathbb{E}|J^{21}_k|^2\lesssim\delta\ep^{2\theta}\leq\delta.
		\end{equation}
		For $J^{22}_k,$ by It\^{o} isometry,
		\begin{eqnarray}
			&&\mathbb{E}|J^{22}_k|^2\nonumber\\&=&\mathbb{E}\Big|\int_{\tau}^{\tau+\delta}\int_{h}^{\tau+\delta}\ep^{\theta-1}e^{-\ep^{-1}u}e^{\ep^{-1}h}du\tilde{L}(h)\Big|^2\nonumber\\&=&\mathbb{E}\int_{\tau}^{\tau+\delta}\int_{|x|< 1}\Big|\int_{h}^{\tau+\delta}x\ep^{\theta-1}e^{-\ep^{-1}u}e^{\ep^{-1}h}du\Big|^2\nu_1(dx)dh\nonumber\\&=&\int_{|x|< 1}|x|^2\nu_1(dx)\mathbb{E}\int_{\tau}^{\tau+\delta}\Big|\int_{h}^{\tau+\delta}\ep^{\theta-1}e^{-\ep^{-1}u}e^{\ep^{-1}h}du\Big|^2dh.
			\end{eqnarray}
		By a direct calculation,$$\int_{h}^{\tau+\delta}\ep^{\theta-1}e^{-\ep^{-1}u}e^{\ep^{-1}h}du\leq\ep^\theta,$$so
		\begin{equation}\label{tightJ22k}
			\mathbb{E}|J^{22}_k|^2\lesssim\int_{|x|< 1}|x|^2\nu_1(dx)\ep^{2\theta}\delta\lesssim\delta.
		\end{equation}
		Substituting (\ref{tightJ21k}) and (\ref{tightJ22k}) into (\ref{tightJ2kready}),
		\begin{equation*}
			\mathbb{E}|J^2_k(\tau+\delta)-J^2_k(\tau)|^2\lesssim\lambda_k^2\delta.
		\end{equation*}
		By Jensen inequality,
		\begin{equation}\label{tightJ2k}
			\mathbb{E}|J^2_k(\tau+\delta)-J^2_k(\tau)|\lesssim\lambda_k\sqrt{\delta}.
		\end{equation}
		Substituting (\ref{tightJ1k}) and (\ref{tightJ2k}) into (\ref{tightJkready})  \begin{equation}\label{tightJk}
			\mathbb{E}|J_k(\tau+\delta)-J_k(\tau)|\lesssim\lambda_k\sqrt{\delta},
		\end{equation}
		and together with (\ref{tightIk}), we finished the proof. 
		
		\end{proof}
		\begin{coro}\label{aldousforue}
			For each stopping time $\tau$ with $\tau+\delta\leq T,$
			$$\mathbb{E}||\ue(\tau+\delta)-\ue(\tau)||_{-1}\lesssim \sqrt{\delta}.$$
		\end{coro}
		\begin{proof}
			By Lemma \ref{aldousforuek},
			\begin{eqnarray*}
				&&\mathbb{E}||(\uek(\tau+\delta)-\uek(\tau))e_k||_{-1}\\&=&\al_k^{-1/2}\mathbb{E}|\uek(\tau+\delta)-\uek(\tau)|\\&\lesssim&\al_k^{1/2}(\al_k^{1/2}\lambda_k+\lambda_k+\frac{1}{k^{1+\gamma}}+k^{1+\gamma}\lambda_k^2)\delta+\al_k^{-1/2}\lambda_k\sqrt{\delta}.
			\end{eqnarray*}
By triangular inequality and $(\mathbf{A_1})$,
$$\mathbb{E}||(\ue(\tau+\delta)-\ue(\tau))||_{-1}\lesssim\delta\sum\limits_{k\in\mathbb{N}}\al_k^{1/2}(\al_k^{1/2}\lambda_k+\lambda_k+\frac{1}{k^{1+\gamma}}+k^{1+\gamma}\lambda_k^2)+\sqrt{\delta}\sum\limits_{k\in\mathbb{N}}\lambda_k\lesssim\sqrt{\delta}.$$
		\end{proof}
		
		\begin{lemma}\label{aldousforre}
			For each stopping time $\tau$ with $\tau+\delta\leq T,$
			$$\mathbb{E}||\re(\tau+\delta)-\re(\tau)||_{-1}\lesssim \delta.$$
		\end{lemma}
		\begin{proof}
			Set $\xe:=\re_t$, so that
			\begin{equation}\label{eqforre2}
				\begin{cases}
					\re_t=\xe,\enspace \re(0)=u_0\\
					\xe_t=\frac{1}{\ep}[-\xe+\Delta\re+f(\ube)],\enspace \xe(0)=v_0.
				\end{cases}
			\end{equation}
			From (\ref{eqforre2}), for all $\psi\in H_0^1(D)$ with $||\psi||_1=1$,
			$$(\xe(t),\psi)-(v_0,\psi)=\frac{1}{\ep}\Big[\intot-(\xe(s),\psi)-(\nabla\re(s),\nabla\psi)+(f(\ube(s)),\psi)ds\Big],$$ so that
			$$(\xe(t),\psi)=(v_0,\psi)+\frac{1}{\ep}e^{-\ep^{-1}t}\intot e^{\ep^{-1}s}[-(\nabla\re(s),\nabla\psi)+(f(\ube(s),\psi))]ds.$$
			Similar to (\ref{fubiniargument}), a use of Fubini theorem yields
			\begin{eqnarray*}
				&&(\re(t),\psi)-(u_0,\psi)\\&=&\intot(\xe(s),\psi)ds\\&=&(v_0,\psi)t-\intot(\nabla\re(r),\nabla\psi)+(f(\ube(r)),\psi)dr-\intot e^{-\ep^{-1}(t-r)}[-(\nabla\re(r),\nabla\psi)+(f(\ube(r)),\psi)]dr,
			\end{eqnarray*}
			so that
			\begin{eqnarray}\label{differenceofre}
				&&|(\re(\tau+\delta),\psi)-(\re(\tau),\psi)|\nonumber\\&=&\Big|(v_0,\psi)\delta+\int_{\tau}^{\tau+\delta}-(\nabla\re(r),\nabla\psi)+(f(\ube(r)),\psi)dr\nonumber\\&&-\Big(e^{-\ep^{-1}(\tau+\delta)}\int_{0}^{\tau+\delta}e^{-\ep^{-1}(t-r)}[-(\nabla\re(r),\nabla\psi)+(f(\ube(r),\psi))]dr\nonumber\\&&-e^{-\ep^{-1}\tau}\int_{0}^{\tau}e^{-\ep^{-1}(t-r)}[-(\nabla\re(r),\nabla\psi)+(f(\ube(r),\psi))]dr\Big)\Big|\nonumber\\&\leq&|(v_0,\psi)|\delta+\Big|\int_{\tau}^{\tau+\delta}-(\nabla\re(r),\nabla\psi)+(f(\ube(r)),\psi)dr\Big|\nonumber\\&&+\Big|e^{-\ep^{-1}(\tau+\delta)}\int_{0}^{\tau+\delta}e^{\ep^{-1}s}[-(\nabla\re(r),\nabla\psi)+(f(\ube(r),\psi))]dr\nonumber\\&&-e^{-\ep^{-1}(\tau+\delta)}\int_{0}^{\tau}e^{\ep^{-1}r}[-(\nabla\re(r),\nabla\psi)+(f(\ube(r),\psi))]dr\Big|\nonumber\\&&+\Big|e^{-\ep^{-1}(\tau+\delta)}\int_{0}^{\tau}e^{\ep^{-1}r}[-(\nabla\re(r),\nabla\psi)+(f(\ube(r)),\psi)]dr\nonumber\\&&-e^{-\ep^{-1}\tau}\int_{0}^{\tau}e^{\ep^{-1}r}[-(\nabla\re(r),\nabla\psi)+(f(\ube(r)),\psi)]dr\Big|\nonumber\\&=:&|(v_0,\psi)|\delta+A^\psi+B^\psi+C^\psi.
			\end{eqnarray}
			By Lipschitz property of $f$ and Cauchy-Schwarz inequality,
			\begin{eqnarray}\label{esforApsi}
				&&A^\psi\nonumber\\&\leq&\int_{\tau}^{\tau+\delta}|-(\nabla\re(r),\nabla\psi)+(f(\ube(r),\psi))|dr\nonumber\\&\lesssim&\int_{\tau}^{\tau+\delta}\supt||\nabla\re(t)||\cdot||\nabla\psi||+(1+\supt||\ube(t)||)||\psi||dr\nonumber\\&\lesssim&\int_{\tau}^{\tau+\delta}(1+\supt||\nabla\re(t)||+\supt||\ube(t)||)||\psi||_1\delta\nonumber\\&\lesssim&(1+\supt||\nabla\re(t)||+\supt||\ube(t)||)\delta,
			\end{eqnarray}
			\begin{eqnarray}\label{esforBpsi}
				&&B^\psi\nonumber\\&=&\Big|e^{-\ep^{-1}(\tau+\delta)}\int_{\tau}^{\tau+\delta}e^{\ep^{-1}r}[-(\nabla\re(r),\nabla\psi)+(f(\ube(r),\psi))]dr\Big|\nonumber\\&\leq&\int_{\tau}^{\tau+\delta}|-(\nabla\re(r),\nabla\psi)+(f(\ube(r)),\psi))|dr\nonumber\\&\lesssim&\int_{\tau}^{\tau+\delta}\supt||\nabla\re(t)||\cdot||\nabla\psi||+(1+\supt||\ube(t)||)||\psi||dr\nonumber\\&\lesssim&(1+\supt||\re(t)||+\supt||\ube(t)||)||\psi||_1\delta\nonumber\\&\lesssim&(1+\supt||\re(t)||+\supt||\ube(t)||)\delta,
			\end{eqnarray}
			and
			\begin{eqnarray}\label{esforCpsi}
				&&C^\psi\nonumber\\&=&|e^{-\ep^{-1}(\tau+\delta)}-e^{-\ep^{-1}\tau}|\Big|\int_{0}^{\tau}e^{\ep^{-1}\tau}[-(\nabla\re(r),\nabla\psi)+(f(\ube(r)),\psi)]dr\Big|\nonumber\\&\lesssim&e^{-\ep^{-1}\tau}(1-e^{-\ep^{-1}\delta})\int_{0}^{\tau}e^{\ep^{-1}r}(\supt||\nabla\re(t)||\cdot||\nabla\psi||+(1+\supt||\ube(t)||)||\psi||)\nonumber\\&\lesssim&e^{-\ep^{-1}\tau}\ep^{-1}\delta\int_{0}^{\tau}e^{\ep^{-1}r}(1+\supt||\nabla\re(t)||+\supt||\ube(t)||)||\psi||_1dr\nonumber\\&\leq&(1+\supt||\nabla\re(t)||+\supt||\ube(t)||)\delta.
			\end{eqnarray}
		Substituting (\ref{esforApsi})-(\ref{esforCpsi}) into (\ref{differenceofre}),
		$$|(\re(\tau+\delta)-\re(\tau),\psi)|\lesssim(1+\supt||\nabla\re(t)||+\supt||\ube(t)||)\delta.$$
		The result follows by taking expectation and using (\ref{esforre}) and Proposition \ref{Hnormforub}.
		\end{proof}
		\begin{prop}
			$(\ube)_{0<\ep\leq1}$ is tight in $\mathbb{D}([0,T],H^{-1}(D))$.
		\end{prop}
		\begin{proof}
		By \cite[Theorem 7]{KR22} and the Markov's inequality, it is sufficient to prove the following two inequalities:
		
		(i) $\supe\mathbb{E}\supt||\ube(t)||<\infty.$
		
		(ii) There exists a function $f$ which is defined on $\mathbb{R}_+$ and which is independent of $\ep,$ such that for each $\delta>0$ and stopping time $\tau$ with $\tau+\delta\leq T$, $$\mathbb{E}||\ube(\tau+\delta)-\ube(\tau)||_{-1}\leq f(\delta),$$ and $f(\delta)\downarrow 0$ as $\delta\downarrow 0.$
		
		Note that (i) is precisely the statement of Proposition \ref{Hnormforub}, and
		since $\ube=\ue+\re,$ (ii) follows by combining Corollary \ref{aldousforue} and Lemma \ref{aldousforre}.
		
		\end{proof}
		\section{Proof of the Main Result}\label{proofofthemainresult}
		\emph{Proof of Theorem \ref{mainresult}:}
		
		Given a test function $\phi\in C^2([0,T]\times D)$ with $\phi|_{\pa D}=0$ by (\ref{decomposition}),
		\begin{eqnarray}
			&&(\ube(t),\phi(t))-(u_0,\phi(0))\nonumber\\&=&\intot(\vbe(s),\phi(s))ds+\intot(\ube(s),\phi_t(s))ds\nonumber\\&=&\ep^{-1}\intot(\bvbe_1(s),\phi(s))ds+\intot(\bvbe_2(s),\phi(s))ds+\ep^{\theta+\frac{1}{\al}-1}\intot(\bvbe_3(s),\phi(s))ds+\intot(\ube(s),\phi_t(s))ds\nonumber.
		\end{eqnarray}
		But from (\ref{bveq}),
		\begin{eqnarray}
			&&(\bvbe_2(t),\phi(t))\nonumber\\&=&-\ep^{-1}\intot(\bvbe_2(s)-f(\ube(s)),\phi(s))ds-\ep^{-1}\intot(\nabla\ube(s),\nabla\phi(s))ds+\intot(\bvbe_2(s),\phi_t(s))ds\nonumber\\&=&-\ep^{-1}\intot(\bvbe_2(s),\phi(s))ds-\ep^{-1}\intot(\nabla\ube(s),\nabla\phi(s))ds+\ep^{-1}\intot(f(\ube(s),\phi(s)))ds\nonumber\\&&+\intot(\bvbe_2(s),\phi_t(s))ds.\nonumber
		\end{eqnarray}
		Combining the two equalities above,
		\begin{eqnarray}
			&&(\ube(t),\phi(t))-(u_0,\phi(0))\nonumber\\&=&\intot(\ube(s),\phi_t(s))ds-\intot(\nabla\ube(s),\nabla\phi(s))ds+\ep^{-1}\intot(\bvbe_1(s),\phi(s))ds\nonumber\\&&+\ep\Big[\intot(\bvbe_2(s),\phi_t(s))ds-(\bvbe_2(t),\phi(t))\Big]+\intot (f(\ube(s)),\phi(s))ds\nonumber\\&&+\ep^{\theta+\frac{1}{\al}-1}\intot(\bvbe_3(s),\phi(s))ds.
		\end{eqnarray}
		Rearranging the equality above,
		\begin{eqnarray}\label{generalestimateforR}
			&&R^\ep(t)\nonumber\\&:=&(\ube(t),\phi(t))-(u_0,\phi(0))-\intot(\ube(s),\phi_t(s))ds\nonumber\\&&+\intot(\nabla\ube(s),\nabla\phi(s))ds-\intot (f(\ube(s)),\phi(s))ds-\ep^{\theta}\intot(\phi(s),dL(s))\nonumber\\&=&\ep^{-1}\intot(\bvbe_1(s),\phi(s))ds+\ep\Big[\intot(\bvbe_2(s),\phi_t(s))ds-(\bvbe_2(t),\phi(t))\Big]\nonumber\\&&+\Big[\ep^{\theta+\frac{1}{\al}-1}\intot(\bvbe_3(s),\phi(s))ds-\ep^\theta\intot(\phi(s),dL(s))\Big]\nonumber\\&=:&J_1(t)+J_2(t)+\ep^\theta J_3(t).
		\end{eqnarray}
		We solve $\bvbe_1$ from (\ref{bveq}) as
		$$\bvbe_1(t)=\ep e^{-\ep^{-1}t}v_0,$$ so that
		\begin{eqnarray}\label{noname}
			&&\ep^{-1}\intot(\bvbe_1(s),\phi(s))ds\nonumber\\&=&\ep^{-1}\intot(\ep e^{-\ep^{-1}s}v_0,\phi(s))ds\nonumber\\&=&\intot e^{-\ep^{-1}s}(v_0,\phi(s))ds\nonumber\\&\lesssim&\intot e^{-\ep^{-1}s}ds\nonumber\\&\leq&\ep,
		\end{eqnarray}
		that is,
		\begin{equation}\label{esforJ1}
			E|J_1(t)|\lesssim\ep.
		\end{equation}
		Let us deal with $J_2$. From (\ref{bveq}), it is straightforward that for each $\psi\in H_0^1$ with $||\psi||_1=1,$
		$$(\bvbe_2(t),\psi)=\ep^{-1}e^{-\ep^{-1}t}\intot e^{\ep^{-1}s}[-(\nabla\ube(s),\nabla\psi)+(f(\ube(s),\psi))]ds.$$
		By the Lipschitz property of $f$,
		\begin{eqnarray}\label{noname1}
			&&|(\bvbe_2(t),\psi)|\nonumber\\&\leq&\ep^{-1}e^{-\ep^{-1}t}\intot e^{\ep^{-1}s}(||\nabla\ube(s)||\cdot||\nabla\psi||+||f(\ube(s))||\cdot||\psi||)ds\nonumber\\&\lesssim&\ep^{-1}e^{-\ep^{-1}t}\intot e^{\ep^{-1}s}(\supt||\nabla \ube(t)||+1+\supt||\ube(t)||)ds\nonumber\\&\lesssim&\supt||\nabla\ube(t)||+1.
		\end{eqnarray}
		By Fubini theorem and Proposition \ref{Hnormforub},
		\begin{eqnarray*}
			&&\mathbb{E}||\bvbe_2(t)||_{-1}\nonumber\\&=&\mathbb{E}\sup\limits_{||\psi||_1=1}|(\bvbe_2(t),\psi)|\nonumber\\&\lesssim&\mathbb{E}(\supt||\nabla\ube(t)||+1)\\&\lesssim&1,
		\end{eqnarray*}
		from which
		\begin{eqnarray}
			&&\mathbb{E}\Big|\intot(\bvbe_2(s),\phi_t(s))ds\Big|\nonumber\\&\leq&\mathbb{E}\intot||\bvbe_2(s)||_{-1}||\phi_t(s)||_1ds\nonumber\\&\lesssim&\mathbb{E}\intot||\bvbe_2(s)||_{-1}ds\nonumber\\&=&\intot\mathbb{E}||\bvbe_2(s)||_{-1}ds\nonumber\\&\leq&T\supt\mathbb{E}||\bvbe_2(t)||_{-1}\nonumber\\&\lesssim&1,
		\end{eqnarray}
		and
		\begin{eqnarray}
			&&\mathbb{E}|(\bvbe_2(t),\phi(t))|\nonumber\\&\leq&\mathbb{E}||\bvbe_2(t)||_{-1}||\phi(t)||_1\nonumber\\&\lesssim&\supt\mathbb{E}||\bvbe_2(t)||_{-1}\nonumber\\&\lesssim&1,
		\end{eqnarray}so
		\begin{equation}\label{esforJ2}
			\mathbb{E}|J_2(t)|\lesssim\ep.
		\end{equation}
		By (\ref{decomposition}).
		\begin{eqnarray}
			&&(\bvbe_3(t),\phi(t))\nonumber\\&=&\intot(\bvbe_{3,t}(s),\phi(s))ds+\intot(\bvbe_3(s),\phi_t(s)) ds\nonumber\\&=&-\ep^{-1}\intot(\bvbe_3(s),\phi(s))ds+\ep^{-\frac{1}{\al}}\intot(\phi(s),dL(s))+\intot(\bvbe_3(s),\phi_t(s))ds,
		\end{eqnarray}
		which is equivalent to say that
		\begin{eqnarray}
			&&J_3(t)\nonumber\\&=&\ep^{\frac{1}{\al}-1}\intot(\bvbe_3(s),\phi(s))ds-\intot(\phi(s),dL(s))\nonumber\\&=&\ep^{\frac{1}{\al}}\intot(\bvbe_3(s),\phi_t(s))ds-\ep^{\frac{1}{\al}}(\bvbe_3(t),\phi(t)).
		\end{eqnarray}
		Set $\bvbe_{3,k}(t):=(\bvbe_3(t),e_k)$, it is straightforward that
		$$\ep^{\frac{1}{\al}}\bvbe_{3,k}(t)=\lambda_k\intot e^{-\ep^{-1}(t-s)}dL_k(s).$$
		But by \cite[Theorem 3.2]{RW06}
		\begin{eqnarray}
			&&\mathbb{E}\Big|\intot e^{-\ep^{-1}(t-s)}dL_k(s)\Big|\nonumber\\&\lesssim&\Big(\intot e^{-\ep^{-1}\al(t-s)}ds\Big)^{\frac{1}{\al}}\nonumber\\&\lesssim&\ep^\frac{1}{\al},
		\end{eqnarray}implying that
		$$\supe\supt\ep^{\frac{1}{\al}}\mathbb{E}|\bvbe_3(t)|\lesssim\ep^\frac{1}{\al}\sum\limits_{k\in\mathbb{N}}\lambda_k\lesssim\ep^\frac{1}{\al},$$
		so \begin{equation}\label{esforJ3}
			\mathbb{E}|J_3(t)|\lesssim \ep^\frac{1}{\al}.
		\end{equation}
		 By (\ref{generalestimateforR}), (\ref{esforJ1}), (\ref{esforJ2}) and(\ref{esforJ3}),
		\begin{equation}\label{esforR}
			\mathbb{E}|R^\ep(t)|\lesssim\ep+\ep^{\theta+\frac{1}{\al}}.
		\end{equation}
		
		Let us consider the case when $\theta=0$. Denote $\mathbb{D}:=\mathbb{D}([0,T];H^{-1}(D)).$ By\cite{GK96}, in order to prove that $(\ube)$ converges in probability as $\ep\to 0$, it is sufficient to prove that for any subsequences $\{\ep(n)\}$ and $\{\mu(n)\},$ there exists subsequences $\{\ep(n_k)\}$ and $\{\mu(n_k)\}$ such that $(U^{\ep(n_k)},U^{\mu(n_k)})$ converges to a $\mathbb{D}^2$-valued random variable $w=(w_1,w_2)$, and $w$ is supported on the diagonal.
		
		Now given two subsequences $\{\ep(n)\}$ and $\{\mu(n)\},$ since $(\ube)_{0<\ep\leq 1}$ is  tight, by Prokhorov and Skorokhod theorem, there exist
		
		(i) a probability space $(\hat{\Omega},\hat{\mathcal{F}},\hat{P}),$
		
		  (ii) a sequence of $\mathbb{D}^2\times\mathbb{D}([0,T],L^2(D))$-valued random variable $(u_1^k,u_2^k,\hat{L}_k)$ defined on $(\hat{\Omega},\hat{\mathcal{F}},\hat{P}),$
		 
		  (iii) a $\mathbb{D}^2\times\mathbb{D}([0,T],L^2(D))$-valued random variable $(u_1,u_2,\hat{L})$ defined on $(\hat{\Omega},\hat{\mathcal{F}},\hat{P}),$
		 
		 \noindent such that 
		 
		 (i) $(u_1^k,u_2^k,\hat{L}_k)\overset{d}{=}(U^{\ep(n_k)},U^{\mu(n_k)},L)$
		 
		  (ii) $(u_1^k,u_2^k,\hat{L}_k)\to (u_1,u_2,\hat{L})$, $\hat{P}-$ a.s..
		 
		  From (\ref{generalestimateforR}) and (\ref{esforR}) we see that
		\begin{equation}\label{relation}
			(\ube(t),\phi(t))=(u_0,\phi(0))+\intot(\ube(s),\phi_t(s)+\Delta \phi(s))ds+\intot(f(\ube(s),\phi(s)))ds+\intot(\phi(s),dL(s))+R^\ep(t),
		\end{equation}
		and for each $0\leq t\leq T,$ $$\mathbb{E}|R^\ep(t)|\lesssim \ep^{\frac{1}{\al}}.$$
		As in (\ref{generalestimateforR}), set for $i=1,2$ and $k\in\mathbb{N},$ $$R^k_i(t):=(u^k_i(t),\phi(t))-(u_0,\phi(0))-\intot(u^k_i(s),\phi_t(s)+\Delta \phi(s))ds-\intot(f(u^k_i(s)),\phi(s))ds-\intot(\phi(s),d\hat{L}_k(s)).$$
		Since $(u_1^k,u_2^k,\hat{L}_k)\overset{d}{=}(U^{\ep(n_k)},U^{\mu(n_k)},L),$ we have
		$$\mathbb{E}|R^k_1(t)|\lesssim[\ep(n_k)]^\frac{1}{\al},\enspace \mathbb{E}|R^k_2(t)|\lesssim[\mu(n_k)]^\frac{1}{\al},$$
		and in particular,
		$$\lim\limits_{k\to\infty}\mathbb{E}|R^k_i(t)|=0,\enspace i=1,2.$$
		Set $\mathcal{D}_i:=\{t\in[0,T]:P(u_i(t)\neq u_i(t-))\}$ and $\mathcal{D}:=\mathcal{D}_1\cup\mathcal{D}_2$. From the c\`{a}dl\`ag property of $u_i$, $\mathcal{D}$ is at most countable\cite[Lemma 7.7, Chapter 3]{EK09}. For each $t\notin\mathcal{D},$ $u^k_i(t)\to u_i(t),$ $\tilde{P}-a.s.,$ so $(u^k_i(t),\phi(t))\to (u_i(t),\phi(t))$, $\tilde{P}-a.s..$ Together with dominated convergence theorem, we conclude that for each $t\notin\mathcal{D},$
		\begin{equation}\label{heatequation}
			(u_i(t),\phi(t))=(u_0,\phi(0))+\intot(u_i(s),\phi_t(s)+\Delta u_i(s))ds+\intot(f(u_i(s)),\phi(s))ds+\intot(\phi(s),dL(s)),
		\end{equation}
		$\hat{P}-\text{a.s.}$, $i=1,2.$ By the c\`adl\`ag property, equation (\ref{heatequation}) holds for each $t\in[0,T],$ almost surely. Therefore, $u_1=u_2$ due to the uniqueness of stochastic heat equation. Therefore, $(\ube)$ converges in probability to some $\bar{U}$ in $\mathbb{D}$ due to\cite{GK96}. But by the same argument we derive that $\bar{U}$ again satisfies (\ref{heatequation}), and the proof of part (i) of Theorem \ref{mainresult} is finished.
		
		Now we turn to prove part (ii). Fix a test function $\psi\in H_0^2(D)$ such that $||\psi||_2=1.$ From Definition \ref{defofsolution}, 
		\begin{equation*}
			(\bube(t),\psi)-(u_0,\psi)=\intot(\bube(s),\Delta\psi)ds+\intot(f(\bube(s)),\psi)ds+\ep^\theta\sum_{k=1}^{\infty}(\psi,e_k)L_k(t).
		\end{equation*}
		From (\ref{generalestimateforR}) we see that
		\begin{equation*}
			(\ube(t),\psi)-(u_0,\psi)=\intot(\ube(s),\Delta\psi)ds+\intot(f(\ube(s)),\psi)ds+\ep^\theta\sum_{k=1}^{\infty}(\psi,e_k)L_k(t)+R^{\ep,\psi}(t),
		\end{equation*}
		where
		\begin{equation}\label{defofReppsi}
			R^{\ep,\psi}(t)=\ep^{-1}\intot(\bvbe_1(s),\psi)ds-\ep(\bvbe_2(t),\psi)+\ep^\theta\Big[\ep^{\frac{1}{\al}-1}\intot(\bvbe_3(s),\psi)ds-(\psi,L(t))\Big].
		\end{equation}
		From the two equalities above,
		$$(\ube(t)-\bube(t),\psi)=\intot(\ube(s)-\bube(s),\Delta\psi)ds+\intot(f(\ube(s))-f(\bube(s)),\psi)ds+R^{\ep,\psi}(t).$$
		Since $f$ is Lipschitz,
		\begin{eqnarray*}
			&&|(\ube(t)-\bube(t),\psi)|\\&\leq&|R^{\ep,\psi}(t)|+\intot||\ube(s)-\bube(s)||_{-2}||\psi||_2ds+\intot||f(\ube(s))-f(\bube(s))||_{-2}||\psi||_2ds\\&\lesssim&|R^{\ep,\psi}(t)|+\intot||\ube(s)-\bube(s)||_{-2}ds.
		\end{eqnarray*}
		Taking supremum,
		\begin{eqnarray*}
			&&\supt\sup\limits_{||\psi||_2=1}|(\ube(t)-\bube(t),\psi)|\\&\lesssim&\supt\sup\limits_{||\psi||_2=1}|R^{\ep,\psi}(t)|+\int_{0}^{T}||\ube(s)-\bube(s)||_{-2}||\psi||_2ds\\&\leq&\supt\sup\limits_{||\psi||_2=1}|R^{\ep,\psi}(t)|+\int_{0}^{T}\sup\limits_{0\leq s\leq t}||\ube(s)-\bube(s)||_{-2}dt.
		\end{eqnarray*}
		By Fubini theorem,
		$$\mathbb{E}\supt||\ube(t)-\bube(t)||_{-2}\lesssim\mathbb{E}\supt\sup\limits_{||\psi||_2=1}|R^{\ep,\psi}(t)|+\int_{0}^{T}\mathbb{E}\sup\limits_{0\leq s\leq t}||\ube(s)-\bube(s)||_{-2}dt.$$
		Our proof is finished by Gronwall inequality if we can prove
		\begin{equation}\label{star}
			\mathbb{E}\supt\sup\limits_{||\psi||_2=1}|R^{\ep,\psi}(t)|\lesssim\ep^\theta.
		\end{equation}
		Now we prove (\ref{star}). Indeed, by the same procedure leading to (\ref{noname}) and (\ref{noname1}), we derive that
		\begin{equation}\label{star1}
			\mathbb{E}\supt\sup\limits_{||\psi||_2=1}\Big|\ep^{-1}\intot(\bvbe_1(s),\psi)ds\Big|\lesssim\ep,
		\end{equation}
		and
		\begin{equation}\label{star2}
			\mathbb{E}\supt\sup\limits_{||\psi||_2=1}|(\bvbe_2(t),\psi)|\lesssim1.
		\end{equation}
		By (\ref{bveq}),
		$$(\bvbe_3(t),\psi)=-\ep^{-1}\intot(\bvbe_3(s),\psi)ds+\ep^{-\frac{1}{\al}}(L(t),\psi).$$
		Multiplying $\ep^{\frac{1}{\al}}$ and rearranging,
		$$\Big|\ep^{\frac{1}{\al}-1}\intot(\bvbe_3(s),\psi)ds-(\psi,L(t))\Big|=|(\ep^{\frac{1}{\al}}\bvbe_3(t),\psi)|.$$
		But from (\ref{esforrho}) and assumption ($\mathbf{A_1}$), $$\mathbb{E}\supt||\ep^{\frac{1}{\al}}\bvbe_3(t)||\lesssim1,$$
		so that
		\begin{eqnarray}\label{star3}
			&&\mathbb{E}\supt\sup\limits_{||\psi||_2=1}\Big|\ep^{\frac{1}{\al}-1}\intot(\bvbe_3(s),\psi)ds-(\psi,L(t))\Big|\nonumber\\&\leq&\mathbb{E}\supt\sup\limits_{||\psi||_2=1}(||\ep^{\frac{1}{\al}}\bvbe_3(t)||\cdot||\psi||)\nonumber\\&\leq&\mathbb{E}\supt||\ep^{\frac{1}{\al}}\bvbe_3(t)||\nonumber\\&\lesssim&1.
		\end{eqnarray}
	(\ref{star}) is proved by substituting (\ref{star1}), (\ref{star2}) and (\ref{star3}) into (\ref{defofReppsi}). The proof of Theorem \ref{mainresult} is finished.\qed
		\appendix
		\section{Appendix}\label{appen}
		In this appendix, $\ep$ is fixed, so we drop the superscript $\ep,$ and we write $x\lesssim y$ to imply that $x\leq Cy$ for some constant $C$ which may depend on $\ep$.
		
		\noindent\emph{Proof of Proposition \ref{wellposedness}:}

		    We only prove the proposition for equation (\ref{wave}), and the proof for equation (\ref{limitequation}) is similar and easier. Without loss of generality, we assume that $\theta=0$. As in the proof of Proposition \ref{Hnormforub}, we firstly consider the linear part, that is 
		   \begin{equation}\label{linearpart}
		   	\begin{cases}
		   		\ep u_{tt}+u_t=\Delta u+\dot{L},\\
		   		u(0)=0, u_t(0)=0.
		   	\end{cases}
		   \end{equation}  
		   We claim that the solution of (\ref{linearpart})  in the sense of (\ref{defofSWE}) is existent, so that for each $\phi\in C^1([0,T]\times D)$ with $\phi|_{\pa D}=0,$ the following equalities hold in the sense that they are indistinguishable:
		   \begin{equation}\label{weaklienarpart}
		   	\begin{cases}
		   		(u(t),\phi(t))=\intot (u(s),\phi_t(s))+(v(s),\phi(s))ds,\\
		   		(v(t),\phi(t))=\frac{1}{\ep}\Big[\intot-(v(s),\phi(s))-(\nabla u(s),\nabla\phi(s))\Big]+\intot(v(s),\phi_t(s))ds\\+\frac{1}{\ep}\sum\limits_{k=1}^{\infty}\lambda_k\intot\phi_k(s)dL_k(s),
		   	\end{cases}
		   \end{equation}
		   where $\phi_k:=(\phi,e_k).$ 
		   
		   Indeed, the linearity implies that the solution of (\ref{linearpart}), if exists, is unique. Besides, we can construct its solution by summing up all its component. For each $k\in\mathbb{N}$, consider the $k$-th component of (\ref{linearpart})
		   \begin{equation}\label{kthcomponent}
		   	\begin{cases}
		   		\dot{u}_k=v_k,\enspace u_k(0)=0,\\
		   		\dot{v}_k=\frac{1}{\ep}(-v_k-\al_ku_k)+\ep^{-1}\lambda_k\dot{L}_k,\enspace v_k(0)=0.
		   	\end{cases}
		   \end{equation}
		   By \cite[Theorem 6.2.9]{APPL09}, equations (\ref{kthcomponent}) admit a unique strong solution, which is predictable and c\`adl\`ag. Set 
		   \begin{equation}\label{constructionofu}
		   	u:=\sum_{k=1}^{\infty}u_ke_k,\enspace v:=\sum_{k=1}^{\infty}v_ke_k.
		   \end{equation}
		   We claim that the first series above converges in $H_0^1(D)$ uniformly on $[0,T]$ almost surely, and the second one converges uniformly on $[0,T]$ in $L^2(D)$ almost surely. Indeed, by Lemma \ref{H-1norm} and $(\mathbf{A_1})$,
		   $$\sum\limits_{k=1}^{\infty}\mathbb{E}\supt||u_k(t)e_k||_1=\sum\limits_{k=1}^{\infty}\al_k^{1/2}\mathbb{E}\supt|u_k(t)|<\infty,$$ so 
		   $$\sum\limits_{k=1}^{\infty}\supt||u_k(t)e_k||_1<\infty, \enspace\text{a.s.}.$$
		   By Cauchy principle, for $\delta>0,$ there exists a positive integer $N$, such that for all $m,n>N,$ $$\sum\limits_{k=m}^{n}\supt||u_k(t)e_k||_1<\delta,\enspace\text{a.s.},$$ so
		   $$\supt\sum\limits_{k=m}^{n}||u_k(t)e_k||_1\leq\sum\limits_{k=m}^{n}\supt||u_k(t)e_k||_1<\delta,\enspace\text{a.s.}.$$
		   By Cauchy principle again, we conclude that $\sum\limits_{k=1}^{\infty}u_k(\cdot)e_k$ converges in $H_0^1(D)$ uniformly on $[0,T],$ a.s.. By a similar argument using Cauchy principle, in order to show the uniform convergence of the series $\sum\limits_{k=1}^{\infty}v_k(\cdot)e_k,$ it suffices to prove $$\sum\limits_{k=1}^{\infty}\mathbb{E}\supt||v_k(t)e_k||=\sum\limits_{k=1}^{\infty}\mathbb{E}\supt|v_k(t)|<\infty.$$
		    Recall from (\ref{compodecompo2}) that $$v_k=w^k_1+\ep^{\theta+\frac{1}{\al}-1}w^k_2.$$
		    Applying Lemma \ref{H1norm}, for each $0\leq t\leq T,$
		    \begin{eqnarray*}
		    	&&|w^k_1(t)|\\&=&\Big|-\ep^{-1}e^{-\ep^{-1}t}\intot\al_k e^{\ep^{-1}s}u_k(s)ds\Big|\\&\leq&\ep^{-1}e^{-\ep^{-1}t}\intot\al_k e^{\ep^{-1}s}|u_k(s)|ds\\&\leq&\ep^{-1}\intot\al_k|u_k(s)|ds\\&\leq&\ep^{-1}\al_k\int_0^T|u_k(s)|ds,
		    \end{eqnarray*}
		    so $$\mathbb{E}\supt|w^k_1(t)|\leq\ep^{-1}\al_k\int_0^T\mathbb{E}|u_k(s)|ds\lesssim\al_k^{1/2}\lambda_k.$$
		    Recall that $w^k_2=\ep^{-\frac{1}{\al}}w^k_3,$ so by (\ref{esforrho}),
		    $$\mathbb{E}\supt|w^k_2(t)|\lesssim\frac{1}{k^{1+\gamma}}+\lambda_k+k^{1+\gamma}\lambda_k^2.$$
		    By the two estimates above, we conclude that
		    $$\mathbb{E}\supt|v_k(t)|\lesssim\al_k^{1/2}\lambda_k+\frac{1}{k^{1+\gamma}}+\lambda_k+k^{1+\gamma}\lambda_k^2.$$ As a consequence of ($\mathbf{A_1}$),
		    $$\sum\limits_{k=1}^{\infty}\mathbb{E}\supt|v_k(t)|<\infty,$$ so that the uniform convergence of $\sum\limits_{k=1}^{\infty}v_k(\cdot)e_k$ in $L^2(D)$ on $[0,T]$ follows.

		   Now we verify that $(u,v)$ constructed above is indeed the weak solution of (\ref{wave}). Applying integration-by-part formula,
		   \begin{eqnarray}\label{finiteweakforv}
		   	&&v_k(t)\phi_k(t)\nonumber\\&=&\intot v_k(s)\dot{\phi}_k(t)ds-\frac{1}{\ep}\intot\phi_k(s)v_k(s)ds\nonumber\\&&-\frac{\al_k}{\ep}\intot\phi_k(s)u_k(s)ds+\frac{1}{\ep}\lambda_k\intot\phi_k(s)dL_k(s).
		   \end{eqnarray}
		   Summing up with respect to $k$ and applying Parseval identity and dominated convergence theorem, 
		   \begin{eqnarray}\label{weakforv}
		   	&&(v(t),\phi(t))\nonumber\\&=&\intot(v(s),\phi_t(s))ds-\frac{1}{\ep}\intot(v(s),\phi(s))ds\nonumber\\&&-\frac{1}{\ep}\intot(\nabla u(s),\nabla\phi(s))ds+\frac{1}{\ep}\sum_{k=1}^{\infty}\lambda_k\intot\phi_k(s)dL_k(s).
		   \end{eqnarray}
		   By the same argument, one verifies
		   \begin{eqnarray}\label{weakforu}
		   	(u(t),\phi(t))=\intot(v(s),\phi(s))+(u(s),\phi_t(s))ds.
		   \end{eqnarray}
		   (\ref{weakforv}) and (\ref{weakforu}) together imply that $(u,v)$ is a solution of the linear SWE (\ref{linearpart}). 
		   
		   It remains to show that the trajectory of $(u,v)$ belongs to $\mathbb{C}([0,T];H_0^1(D))\times\mathbb{D}([0,T]\times L^2(D)),$ but this follows from (\ref{constructionofu}) and the following classical lemma, in conjunction with the observation that $u_k(\cdot)$ is continuous since it is the Lebesgue integral of $v_k(\cdot).$
		   \begin{lemma}\label{realanalysis}
		   	Let $(E,d)$ be a complete metric space and suppose that $(f_k)_{k\in\mathbb{N}}$ is a sequence of $E$-valued continuous (c\`{a}dl\`{a}g) function defined on $[0,T]$. If $f_k$ converges to an $E$-valued function $f$ uniformly on $[0,T]$, then $f$ is also continuous (c\`{a}dl\`{a}g).
		   \end{lemma}
		   
		   Secondly, we solve the pathwise deterministic PDE
		   \begin{equation}\label{pdpde}
		   	\begin{cases}
		   		\ep\rho_{tt}+\rho_t=\Delta\rho+f(u+\rho),\\
		   		\rho(0)=u_0,\enspace \rho_t(0)=v_0.
		   	\end{cases}
		   \end{equation}
		 Set $\rho_t=\xi.$ By a standard Faedo-Galerkin approximation \cite[page 224]{TEMAM}, the above equation has a unique weak solution $(\rho,\xi)\in\mathbb{C}([0,T];H_0^1(D))\times\mathbb{C}([0,T];L^2(D))$, so that
		  \begin{equation}\label{weakpdpde}
		  	\begin{cases}
		  		(\rho(t),\phi(t))-(u_0,\phi(0))=\intot(\rho(s),\phi_t(s))+(\xi(s),\phi(s))ds,\\
		  		(\xi(t),\phi(t))-(v_0,\phi(0))=\frac{1}{\ep}\Big[\intot-(\xi(s),\phi(s))-(\nabla\rho(s),\nabla\phi(s))+(f(u(s)+\rho(s)),\phi(s))ds\Big]\\+\intot(\xi(s),\phi_t(s))ds.
		  	\end{cases}
		  \end{equation}
		  Define $U:=u+\rho,\enspace V:=v+\xi.$ Combining (\ref{weaklienarpart}) and (\ref{weakpdpde}), we see that $(U,V)$ is a weak solution of (\ref{wave}) and that $(U,V)\in\mathbb{C}([0,T];H_0^1(D))\times \mathbb{D}([0,T];L^2(D)).$ The uniqueness of $(U,V)$ follows from the uniqueness of $(u,v)$ and $(\rho,\xi)$.\qed

		\bibliographystyle{abbrv}
		\addcontentsline{toc}{section}{References}
		\bibliography{bibli}	
	
\end{document}